\documentclass[11pt]{amsart}
\usepackage{epsfig}
\usepackage{psfrag}
\usepackage{amsmath}
\usepackage{amssymb}
\usepackage{latexsym}
\usepackage{amscd}
\usepackage{amsfonts}
\usepackage{graphicx}

\usepackage{mathrsfs}
\usepackage{amsfonts}
\usepackage{amsopn}
\usepackage{amstext}
\usepackage{amscd}
\usepackage{amsthm}

\newtheorem{theorem}{Theorem}[section]

\newtheorem{lemma}{Lemma}[section]

\newcommand{\N}{\mathbb{N}}
\newcommand{\R}{\mathbb{R}}

\newcommand{\gap}{\ensuremath{g}}

\DeclareMathOperator{\expect}{E}
\DeclareMathOperator{\prob}{P}

\DeclareMathOperator{\e}{e}

\DeclareMathOperator{\T}{T}

\DeclareMathOperator{\diff}{d}

\newcommand{\comment}[1]{}

\numberwithin{equation}{section}
\begin{document}
\title[Fluctuation of Optimal Alignments]{A Monte Carlo Approach to the Fluctuation Problem in Optimal Alignments of Random Strings}

\author[S.~Amsalu]{Saba Amsalu}
\address{Saba Amsalu,
Department of Information Science, Faculty of Informatics, University of Addis Ababa, Ethiopia}
\email{saba@uni-bielefeld.de}
\thanks{Saba Amsalu was supported by the Engineering and Physical Sciences Research Council [grant number EP/I01893X/1] and by Pembroke College Oxford.}
\author[R.A.~Hauser]{Raphael Hauser} 
\address{Raphael Hauser, 
Mathematical Institute, University of Oxford, 24-29 St Giles', Oxford OX1 3LB, United Kingdom}
\email{hauser@maths.ox.ac.uk}
\thanks{Raphael Hauser was supported by the Engineering and Physical Sciences Research Council [grant number EP/H02686X/1]}
\author[H.F.~Matzinger]{Heinrich Matzinger}
\address{Heinrich Matzinger,
School of Mathematics, Georgia Institute of Technology, 686 Cherry Street, Atlanta, GA 30332-0160 USA}
\email{matzi@math.gatech.edu}
\thanks{Heinrich Matzinger was supported by the Engineering and Physical Sciences Research Council [grant number EP/I01893X/1], IMA Grant SGS29/11, and by Pembroke College Oxford}

\subjclass{Primary 60K35; Secondary 60C05, 60F10, 62E20, 65C05, 90C27}


\keywords{Fluctuation of sequence alignment scores, percolation theory, Monte Carlo simulation, large deviations.}

\begin{abstract}The problem of determining the correct order of fluctuation of the optimal alignment
score of two random strings of length $n$ has been open for several decades. 
It is known \cite{VARTheta} that the biased expected effect of a random letter-change on the optimal score 
implies  an order of fluctuation linear in $\sqrt{n}$. 
However, in many situations where such  a biased effect is observed empirically,
it has  been impossible to prove analytically. The main result of this 
paper shows that when the rescaled-limit of the optimal alignment score increases in a certain direction, 
then the biased effect exists. On the 
basis of this result one can quantify a confidence level for the existence of such a biased effect and hence 
of an order $\sqrt{n}$ fluctuation based on simulation of
 optimal alignments scores.
This is an important step forward, as the correct order of 
fluctuation was previously known 
only for certain special distributions  
\cite{VARTheta},\cite{periodiclcs},\cite{bonettolcs},\cite{increasinglcs}.
To illustrate the usefulness of our new methodology, we apply it to optimal alignments of 
strings 
written in the DNA-alphabet. As scoring function, we use 
 the BLASTZ default-substitution matrix together with
  a realistic gap penalty. BLASTZ is one of the most widely 
used sequence alignment methodologies in bioinformatics. 
For this DNA-setting, we show that with a high level of confidence,
 the fluctuation of the optimal alignment score is of order $\Theta(\sqrt{n})$.
An important special case of optimal alignment score is the Longest Common
Subsequence (LCS) of random strings. For binary sequences with equiprobably symbols
the question of the fluctuation of the LCS remains open. The symmetry in that case
does not allow for our method. On the other hand, in real-life DNA sequences,
it is not the case that all letters occur with the same frequency. So, for many
real life situations, our method allows to determine the order of the fluctuation
up to a high confidence level.
\end{abstract}

\maketitle

\section{Introduction}

Let $x=x_1x_2\ldots x_n$ and $y=y_1y_2\ldots y_n$ be two finite strings written with symbols from a finite alphabet $\mathcal{A}$.
An {\it alignment with gaps} $\pi$ of $x$ and $y$ is 
a strictly increasing integer sequence contained in $[1,n]\times[1,n]$.
Thus,
$$\pi=((\mu_1,\nu_1),(\mu_2,\nu_2),\ldots,(\mu_k,\nu_k))$$
where
$1\leq \mu_1<\mu_2<\ldots<\mu_k\leq n$ and  $1\leq \nu_1<\nu_2<\ldots<\nu_k\leq n$.
The alignment $\pi$ {\em aligns} the symbol $x_{\mu_i}$ with $y_{\nu_i}$ for $i=1,2,\ldots,k$. 
The symbols in the strings $x$ and $y$ that are not aligned with a letter
are said to be {\em aligned with a gap}.
We will use the symbol $\gap$ to denote gaps and write $\mathcal{A}^*=\mathcal{A}\cup\{\gap\}$ for the 
augmented alphabet. 
A {\it scoring function} is a map $S$ from 
$\mathcal{A}^*\times \mathcal{A}^*$ to the set of real numbers.
In everything that follows, we take $S$ to be symmetric so that
$S(c,d)=S(d,c)$ for all $c,d\in \mathcal{A}^*$. 
The {\em alignment score according to $S$} under an alignment $\pi$ of two strings $x$ and $y$ 
is defined as 
$$S_\pi(x,y):=\sum_{i=1}^kS(x_{\mu_i},y_{\nu_i})+\sum_{j\notin \mu}S(x_j,\gap)
+\sum_{j\notin \nu}S(\gap,y_j),$$
where $\mu=\{\mu_1,\ldots,\mu_k\}$ and $\nu:=\{\nu_1,\ldots,\nu_k\}$.

An {\em optimal alignment} of two strings $x$ and $y$ 
is an alignment with gaps that maximizes the alignment score for a given scoring function. 
Note that the set of optimal alignments depends thus not only on $x$ and $y$, but also on the scoring 
function $S$. 

{\footnotesize As an example of an alignment with gaps, let us assume that one species' DNA contains the string
$x=AGTTCG$ and another's the string $y=AATTAC$, where $x$ and $y$ are thought of as 
potentially related. Consider the alignment $\pi$ given 
by the following diagram, 
$$
\begin{array}{c|c|c|c|c|c|c|c|c}
x& &A&G&T&T& &C&G\\\hline
y& &A&A&T&T&A&C&
\end{array}
$$
The alphabet $\mathcal{A}$ we consider in this example is $\mathcal{A}=\{A,T,C,G\}$, and  $\mathcal{A}^*=\mathcal{A}\cup\{\gap\}$ is the augmented alphabet.  
The alignment score under $\pi$ of $x$ and $y$ is given by
$$S_\pi(x,y):=S(A,A)+S(G,A)+S(T,T)+S(T,T)+S(\gap,A)+S(C,C)+S(G,\gap).$$
In this example $\pi$ is an optimal alignment when $S$ assigns a score of $1$ to identical letters and 
a score of $-1$ for two different letters aligned to one other or a letter aligned with a gap. 
}

Alignment scores are widely used in bioinformatics and natural language processing. 
In computational genetics, gaps are interpreted as letters
that disappeared in the course of evolution. The {\em historical alignment} of two DNA-sequences is the alignment with gaps that aligns letters that evolved from the same letter in the common ancestral DNA. This alignment is unknown, but 
if it were available, it would yield information about how closely related two biological species are, how long ago their genomes started to diverge, and what the phylogenetic tree of a chosen set of 
species looks like.   
An important task in bioinformatics is therefore to estimate which 
alignment is most likely to be the ``historic alignment''. 

When the scoring function is the log-likelihood that two letters
evolved from a common ancestral letter, alignments with maximal alignment 
score are also the most likely historic alignments, assuming that 
letters mutate or get deleted independently of their neighbors. This observation is the basis 
for using optimal alignment scores to test whether two sequences are related or not.
Unrelated sequences should be stochastically independent, and this should be reflected by a lower 
optimal alignment score. 
To understand how powerful such a relatedness test is, one needs
to understand the size of the fluctuation of the optimal alignment score, but the fluctuation depends 
of course on the stochastic model used for unrelated DNA sequences and on the scoring function.  

In this paper we consider two finite random strings $X=X_1X_2\ldots X_n$ and 
$Y=Y_1Y_2\ldots Y_n$ of length $n$ in which all letters $X_i$ $(i=1,\dots, n)$ 
and $Y_j$ $(j=1,\dots,n)$ are i.i.d.\ random variables that take values in a given finite alphabet 
$\mathcal{A}$. For any letter $a\in\mathcal{A}$, 
let $p_a$ denote the probability 
$$p_a=P(X_i=a)=P(Y_j=a).$$
Let $S:\mathcal{A}^*\times\mathcal{A}^*\rightarrow \mathbb{R}$ be a scoring function.
We denote the optimal alignment score of $X=X_1\ldots X_n$ and $Y=Y_1\ldots Y_n$ according     to $S$ by 
$$L_n(S):=\max_\pi S_\pi(X,Y)=\max_\pi S_\pi(X_1X_2\ldots X_n,Y_1Y_2\ldots Y_n),$$
where the maximum is taken over all the alignments $\pi$ with gaps aligning $X$ and $Y$. 
Let $\lambda_n(S)$ denote the rescaled expected optimal alignments score
$$\lambda_n(S)=\frac{E[L_n(S)]}{n}.$$
A simple subadditivity argument \cite{Sankoff1} shows that $\lambda_n(S)$ converges
as $n$ goes to infinity. We denote this limit by $\lambda(S)$ and hence
$$\lambda(S):=\lim_{n\rightarrow\infty}\lambda_n(S)
=\lim_{n\rightarrow\infty}\frac{E[L_n(S)]}{n}.$$

The rate of convergence of the last limit above , was bounded by Alexander \cite{Alexander},\cite{alexander2}.
We also give our own bound in the Appendix.

One of the important questions concerning optimal alignments is the asymptotic 
order of fluctuation when $n$ goes to infinity. Although McDiarmid's inequality implies that $VAR[L_n]$ 
is at most of order $O(n)$, it has been a long standing open problem as to whether or not this upper 
bound is tight up to a multiplicative constant, in other words, whether 
\begin{equation}
\label{mainorder}
VAR[L_n(S)]=\Theta(n)
\end{equation}
holds. Steel \cite{Steele86} has proven that for the Longest Common Subsequenc case,
(which is a special case of optimal alignment with the scoring function being
the identity matrix and a zero gap-penalty), one has $VAR[L_n(S_)]\leq n$. The rate of convergence
 Several conflicting conjectures have been proposed about this problem: While Watermann
conjectured \cite{Vingron} that the order is indeed given by \eqref{mainorder},  Chv\`atal and Sankoff conjectured 
a different order \cite{Sankoff1} which would be more in line with corresponding results on 
Last Passage Percolation (LPP)
models, where there exist several situations \cite{Aldous99},\cite{BaikDeiftJohansson99} in which it known that the order 
of fluctuation is the third root of the order of the expectation.

Optimal alignment scores can be reformulated as a LPP problem with correlated weigths.
 We find it  interesting and surprising, that our results are totally
different from the order found in others LPP models. 

In several special cases \cite{VARTheta},\cite{periodiclcs},\cite{increasinglcs}, the  order \eqref{mainorder} 
has been proven 
analytically. In each case the proof was based on the technique of reducing the fluctuation problem to 
the biased effect of a random change in the sequences: In \cite{bonettolcs} and 
\cite{VARTheta} it was established 
that if changing one letter at random has a positive biased effect on $L_n(S)$, 
then the order \eqref{mainorder} must hold. More specifically, for two given
letters $a, b\in\mathcal{A}$, let $(\tilde{X},\tilde{Y})$ denote
the sequence-pair obtained from $(X,Y)$ by changing exactly one entry, chosen uniformly at random among all the letters 
$a$ that appear in $X$ and $Y$, into a $b$.

{\footnotesize Take for example,  $x=aababc$ and $y=abbbbb$. Then,
there are a total of $4$ $a$'s when we count all the $a$'s in both sequences together.
Each of these $a$'s has thus a probability of $1/4$ to get chosen and replaced by
a $b$. Since only one $a$ is changed in both strings $x$ and $y $, we have that
after our letter change one of the strings will remain identical and the other will
be changed by one letter. Let us denote by $\tilde{x}$ and $\tilde{y}$
the sequences after the change. In this example, the $a$ in $y$ has a probability
of $1/4$ to be chosen. If it gets chosen $y$ is transformed into $bbbbbb$.
So, we have $P(\tilde{y}=bbbbbb,\tilde{x}=x)=1/4$. There are $3$ $a$3's in $x$.
So the probability that $x$ get changed is $3/4$. Hence,
$$P(\tilde{x}\in\{bababc,abbabc,aabbbc\} ,\tilde{y}=y)=3/4.$$
}

Let us denote the optimal alignment 
score of $\tilde{X}$ and $\tilde{Y}$ by 
$$\tilde{L}_n(S)=\max_\pi S_\pi(\tilde{X},\tilde{Y}).$$
\noindent In \cite{VARTheta}, it was now shown that if there is a constant $c>0$, 
not depending on $n$, such that 
$$E[\tilde{L}_n(S)-L_n(S)|X,Y]\geq c$$
holds with high probability, then the order \eqref{mainorder} follows. In this context 
``high probability'' is defined as a probability $1-O(n^{-\alpha n})$, for some constant $\alpha>0$ 
that does not depend on $n$. An alternative proof of this result is given in Lemma \ref{alterhut} in the 
next section. 

One of the shortcomings of the above-cited papers \cite{VARTheta}, \cite{bonettolcs} 
is that a strong asymmetry is required 
in the distribution on $\mathcal{A}$ used to generate the random strings $X,Y$ for it to be possible to 
prove the existence of a biased effect of random letter changes. In many situations of relevance to  
applications, the biased effect is visible in simulations but cannot be established analytically using the 
techniques  from \cite{VARTheta}. The present paper addresses this problem: Theorem \ref{samuti} 
establishes that as soon as 
\begin{equation}\label{lambda}
\lambda(S)-\lambda(S-\epsilon T)>0
\end{equation}
for any $\epsilon>0$, the biased effect of a random letter change exists, and this in turn 
implies the fluctuation order \eqref{mainorder}. In this context, let $a$ and $b$ be fixed elements of ${\mathcal A}$, and let $T:\mathcal{A}^*\times\mathcal{A}^*\rightarrow\mathbb{R}$ be the 
scoring function given by 
$T(a,c)=T(c,a):=S(b,c)-S(a,c)$ for all $c\in\mathcal{A}^*$ with $c\neq a$ and $T(d,c)=0$ when $d,c\neq a$.
Furthermore, let $T(a,a):=2(S(b,c)-S(a,c))$.

The practical importance of this result is that the validity of Condition \eqref{lambda} can be verified by Monte Carlo simulation up to any desired confidence level, and this in turn yields a test on whether the order 
\eqref{mainorder} holds, at the same confidence level. A practical example of such a test 
is given in Section \ref{monte}. 

\section{Details of the results}
We will consider strings of length $n$ written with letters from a finite alphabet
$\mathcal{A}$.

{\footnotesize Consider for example, the following 
two strings $x=babbababbba$ and $y=bbbbabbbabb$. We consider {\it alignments with gaps of 
two sequences} This means the letters are aligned with a letter or with a gap. 
Let us see
an example of an alignment with gaps $\pi$  of $x$ with $y$:
\begin{equation}
\label{exampleofalignment}
\begin{array}{c|c|c|c|c|c|c|c|c|c|c|c|c|c}
x& &b&a&b&b&a&b&a&b& &b&b&a\\\hline
y& &b&b&b&b&a&b&b&b&a&b&b&
\end{array}
\end{equation}
Alignment with gaps are used to compare similar sequences. For this purpose one uses
a scoring function $S$ from $\mathcal{A}^*\times\mathcal{A}^*$. Here $\mathcal{A}^*$
represents the alphabet $\mathcal{A}$ augmented by a symbol $G$ representing the gap.
The scoring function should measure how close letters are. The total score 
of an alignment is denoted by $S_\pi(x,y)$. It is the sum of the scores of the 
aligned symbols pairs. In the present example,
the alignment score  for the alignment $\pi$ is equal to:
\begin{equation*}
S_\pi(x,y):=S(b,b)+S(a,b)+2S(b,b)+S(a,a)+S(b,b)+S(a,b)\\
+S(b,b)+S(G,a)+2S(b,b,)+S(a,G).
\end{equation*}

An alignment which  maximizes for given strings $x$ and $y$ the alignment score is called {\it optimal alignment}. Of course which alignment is optimal depends on the scoring function we use. 
We will count the number of aligned symbol pairs appearing in an alignment of $x$ with $y$.
In the example of alignment $\pi$ presently under consideration -- see \eqref{exampleofalignment} -- 
we have $7$ times $b$ aligned with itself. We denote the number of times we see a $b$
aligned with a $b$ by
$Q_\pi(b,b)$. Hence in our example: $Q_\pi(b,b)=7$. In general, for any two 
letters $c,d$ from $\mathcal{A}^*$, let $Q_\pi(c,d)$ be the total number
of columns where  $c$  from $x$ gets aligned with a $d$ from $y$. 
Now, clearly we can write the total alignment score in terms of the values $Q_\pi(c,d)$:
\begin{equation}
\label{}
S_\pi(x,y)=\sum_{c,d\in\mathcal{A}^*}S(c,d)\cdot Q_\pi(c,d)
\end{equation}
We are next going to consider the effect of changing
a randomly chosen $a$ in $x$ or $y$ into a $b$. Among all the $a$'s in
$x$ and $y$ we chose exactly one with equal probability, so that the chosen letter 
will be either in $x$ or in $y$. Let $\tilde{x}$ and $\tilde{y}$ 
denote the sequences $x$ and $y$ after our random letter change.
Note that either $x=\tilde{x}$ or $y=\tilde{y}$, as only one letter changed. 
We want to calculate the expected change:
$$E[S_\pi(\tilde{x},\tilde{y})-S_\pi(x,y)]$$
We find the following formula
\begin{equation}\label{tildeQ*}
E[S_\pi(\tilde{x},\tilde{y})-S_\pi(x,y)]
=\frac{1}{n_a}\sum_{c\in\mathcal{A}^*}\left(Q_\pi(a,c)(S(b,c)-S(a,c))+Q_\pi(c,a)(S(c,b)-S(c,a))
\right)
\end{equation}
where $n_a$ denotes the total number of $a$'s in both strings $x$ and $y$ counted
together. 

To understand formula \eqref{tildeQ*} consider the example of an alignment $\pi$ 
given in \eqref{exampleofalignment} and let us calculate
the expected change in alignment score due to our random change. In $x$
there are two $a$'s which are aligned with a $b$. When any one of them gets chosen
and transformed into $b$, then the change in alignment score is $S(b,b)-S(a,b)$.
This event has probability $2/n_a$. Hence, for our conditional expectation, 
this adds a term $(S(b,b)-S(a,b))\cdot 2/n_a$.
There is also one letter $a$ in $x$ aligned with $a$, the change of which to a $b$ results 
in a change in score of $S(b,a)-S(a,a)$. The probability is $1/n_a$, so
this contributes $(S(b,a)-S(a,a))/n_a$ to the expected change in score.
Finally, there is one $a$ in $y$ which is aligned with an $a$. If this $a$
gets changed to a $b$, then the change is $(S(a,b)-S(a,a))$, which happens
with probability $1/n_a$. The contribution to the expected change from
this letter is thus $(S(a,b)-S(a,a))\cdot(1/n_a)$. Now let us assume that the alignment of a gap
gives the same value whether it is aligned with $a$ or $b$. When we chose an $a$ aligned 
with a gap for our random letter change, the score remains the same.
The contribution of the $a$'s aligned with gaps to the expected change
is thus $0$ in this case. 
Summing up the above contributions,  the expected change of the alignment score in our example 
is equal to
$$E[S(\tilde{x},\tilde{y})]
=\frac{2(S(b,b)-S(a,b))+1\cdot(S(b,a)-S(a,a))+1\cdot(S(a,b)-S(a,a))}{n_a},$$
where  $n_a=6$. Compare the above formula to \ref{tildeQ*}.

If we now define the functional $T$:
$$T:\mathcal{A}^*\times\mathcal{A}^*\rightarrow\mathbb{R}
$$
where for all $c\in\mathcal{A}^*$ where $c\neq a$, we have
$$T(a,c):=S(b,c)-S(a,c)$$ and
$$T(c,a):=S(c,b)-S(c,a)$$ and $T(c,d):=0$ if $d,c\neq a$.
Furthermore, $T(a,a)=2(S(b,a)-S(a,a))$.

Note that since $S$ is symmetric, we also have that $T$ is  symmetric.\\
The expected effect of our random change of letters corresponds
to the ``alignment score according to $T$'' rescaled by the total number
of $a$'s in $x$ and $y$. So equation \ref{tildeQ*} using $T$ becomes
\begin{equation}
\label{tildeQ*2}
E[S_\pi(\tilde{x},\tilde{y})-S_\pi(x,y)]=
\frac{\sum_{c,d\in\mathcal{A}^*} Q_\pi(c,d)\cdot T(c,d)}{n_a}=\frac{T_\pi(x,y)}{n_a}
\end{equation}
where $T_\pi(x,y)$ denote the score of the alignment $\pi$ aligning $x$ with $y$
and using as scoring function $T$ instead of $S$ .
}

Let us next present a theorem which shows that when 
$\lambda(S)-\lambda(S-\epsilon T)>0$, then
a random change of an $a$ into a $b$ has typically a positive biased effect
on the optimal alignment score $L_n(S)$:

\begin{theorem}
\label{samuti}Let $\mathcal{A}$ be a finite alphabet and $S:\mathcal{A}^*\times\mathcal{A}^*\rightarrow\mathbb{R}$ a scoring function. Let the function 
$T:\mathcal{A}^*\times\mathcal{A}^*\rightarrow\mathbb{R}$ defined as above for 
two given letters $a,b$ from ${\mathcal A}$. 
If there exists $\epsilon>0$, such that $\lambda(S)-\lambda(S-\epsilon T)>0$, then for any given constant $\delta>0$ there exists $\alpha>0$ so that the following holds true for all $n$ large enough,
\begin{equation}\label{h2}
P\left(\;E[\tilde{L}_n(S)-L_n(S)|X,Y]\geq 
\frac{\lambda(S)-\lambda(S-\epsilon T)}{\epsilon\cdot p_a}-\delta\;\right) \geq 1-n^{-\alpha\ln(n)},
\end{equation}
where $p_a=P(X_i=a)=P(Y_i=a).$
\end{theorem}

In several instances \cite{VARTheta},\cite{fluctlcsnotsym},
it was proven that when a random
change on the strings has a positive biased expected effect on the score,
then the fluctuation order $VAR[L_n(S)]=\Theta(n)$ applies. 
For the special framework of the current paper, we prove this fact in Lemma \ref{alterhut} below. 
Together with Theorem \ref{samuti} this result implies that if there exists $\epsilon>0$ so that$\lambda(S)-\lambda(S-\epsilon T)>0$, then the fluctuation order \eqref{mainorder} holds, 
see Theorem \ref{canary}.

Section \ref{Theproof} is dedicated to proving Theorem \ref{samuti}.
The main idea behind the proof is quite straightforward and will be briefly explained here: Let 
$X=X_1\ldots X_n$ and $Y=Y_1\ldots Y_n$ as before. From Equation \eqref{tildeQ*2} in the 
example above it follows that for any optimal alignment $\pi$ of $X$ and $Y$ according to $S$,
we have 
\begin{equation}\label{uno}
E[\tilde{L}_n(S)-L_n(S)|X,Y]\geq \frac{T_{\pi}(X,Y)}{N_a}.
\end{equation}
Here $T_\pi(X,Y)$ denotes the alignment score of $\pi$
aligning $X$ and $Y$ according to the scoring function $T$.
Furthermore $N_a$ denotes the total number of $a$'s in $X$ and in $Y$
combined. By linearity of the alignment score, we find that
\begin{equation}\label{duo}
\epsilon\cdot T_{\pi}(X,Y)=S_\pi(X,Y)-(S-\epsilon T)_\pi(X,Y).
\end{equation}
Here $(S-\epsilon T)_\pi(X,Y)$ denotes the score of the alignment $\pi$
aligning $X=X_1\ldots X_n$ and $Y=Y_1\ldots Y_n$ but when we use
the scoring function $(S-\epsilon T)$ instead of $S$. The alignment $\pi$ is optimal 
for $S$ but not necessarily for $(S-\epsilon T)$.
Hence $S_\pi(X,Y)$ is equal to the optimal alignment score $L_n(S)$, but
$(S-\epsilon T)_\pi(X,Y)$ is less or equal to $L_n(S-\epsilon T)$.
This implies that
\begin{equation}\label{tres}
 S_\pi(X,Y)-(S-\epsilon T)_\pi(X,Y)\geq L_n(S)-L_n(S-\epsilon T).
\end{equation}
Combining inequalities \ref{uno}, \ref{duo},\ref{tres},
we obtain
\begin{equation}\label{quattro}
E[\tilde{L}_n(S)-L_n(S)|X,Y]\geq \frac{L_n(S)-L_n(S-\epsilon T)}{n}\frac{n}{\epsilon N_a^x}.
\end{equation}
Note that the right side of the last inequality above converges in probability
to
$$\frac{\lambda(S)-\lambda(S-\epsilon T)}{\epsilon \cdot p_a},$$
where $p_a$ is the probability of letter $a$. This already implies
that the probability on the left side of inequality \ref{h2} in Theorem \ref{samuti} goes
to $1$ as $n\rightarrow\infty$. The rate like in inequality \ref{h2}
can then easily be obtained from the Azuma-Hoeffding Theorem \ref{Azuma} given below. Again the details
of this proof are given in the next section.
Next, let us  formulate the lemma below which shows, that a biased effect of our random
letter change implies the desired order of the fluctuation. We give the proof
because unlike in \cite{VARTheta}, we also consider the case where we have more than
$2$ letters in the alphabet.

\begin{lemma}\label{alterhut} Assume that there exist constants $\Delta>0$ and $\alpha>0$ such that 
for all $n$ large enough it is true that 
\begin{equation}\label{lalala}
P\left(\;
E[\tilde{L}_n(S)-L_n(S)|X,Y]\geq \Delta\;\right)
\geq 1-n^{-\alpha \ln(n)}.
\end{equation}
Then, we have $VAR[L_n(S)]=\Theta(n)$. 
\end{lemma}

\begin{proof} 
Let $N_b$ denote the total number of symbols $b$ in the string $X=X_1X_2\ldots X_n$ and $Y=Y_1Y_2\ldots Y_n$
combined. (This means that we take the number of $b$'s in $X$ and the number of $b$'s
in $Y$ and add them together to get $N_b$).
Note that $N_b$ has a binomial distribution with
$$E[N_b]=2p_b\cdot n,VAR[N_b^x]=4p_b(1-p_b)n,$$
where $p_b:=P(X_i=b)=P(Y_i=b)$.

Let $N_{ab}$ denote the total number of symbols $b$ and $a$'s in the string 
$X=X_1X_2\ldots X_n$ and $Y=Y_1Y_2\ldots Y_n$ combined.
Note that $N_{ab}$ has a binomial distribution with
$$E[N_{ab}]=2(p_a+p_b)\cdot n,$$
where $p_a:=P(X_i=a)=P(Y_i=a)$.

Next we are going to define a collection of random string-pairs 
$(X(k,l),Y(k,l))$ for every $l\leq 2n$ and $k\leq l$. The string-pair
$(X(k,l),Y(k,l))$ has its distribution equal to the string-pair
$$(X,Y)=(X_0X_1\ldots X_n,Y_1Y_2\ldots Y_n)$$
conditional on $N_b=k,N_{ab}=l$. Hence,
$$\mathcal{L}(X(k,l),Y(k,l))=\mathcal{L}(X,Y|N_b=k,N_{ab}=l).$$
For given $l\leq 2n$, we define $(X(k,l),Y(k,l))$ by induction on $k$: 
For this let $(X(0,l),Y(0,l))$ denote a string-pair of length $n$ which is independent
of $N_b$ and of $N_{ab}$. We also, require that $(X(0,l),Y(0,l))$ has its distribution
equal to $(X,Y)$ conditional on $N_b=0$ and $N_{ab}=l$. Then, we chose one $a$
at random\footnote{That is, we chose an $a$ at random among
all $a$'s in $X$ and in $Y$ with equal probability.} 
in $(X(0,l),Y(0,l))$ and change it into a $b$. This yields the string-pair
$(X(1,l),Y(1,l))$. Once $(X(k,l),Y(k,l))$ is obtained, we chose an $a$ at random in 
$(X(k,l),Y(k,l))$ and change it into a $b$. This then give the string-pair
$(X(k+1,l),Y(k+1,l))$. We go on until $k=l$. We do this construction by induction 
on $k$ for every $l=1,2,\ldots,n$.

Now, due to invariance under permutation, we can see that indeed
with this definition we obtain
that 
$$(X(k,l),Y(k,l))$$
has the distribution of $(X,Y)$ given $N_b=k,N_{a,b}=l$.
Hence, $(X(N_b,N_{ab}),Y(N_b,N_{ab}))$ has the same distribution as $(X,Y)$.
So, the optimal alignment score of $X(N_b,N_{ab})$ and $Y(N_b,N_{ab})$
has same distribution as the optimal alignment score of $X$ and $Y$.
Hence, we also have the same variance:
\begin{equation}
\label{VARfN}
VAR[f(N_b,N_{ab})]=VAR[L_n(S)]
\end{equation}
where $f(N_b,N_{ab})$ denotes the optimal alignment score of $X(N_b,N_{ab})$ and $Y(N_b,N_{ab})$.
(In other words, $f(k,l)$ is defined to be the optimal alignment score of $X(k,l)$ and $Y(k,l)$.)
By conditioning, we only can  reduce the variance and hence:
\begin{equation}
\label{VARfN2}
VAR[f(N_b,N_{ab})]\geq E[\; VAR[f(N_b,N_{ab})|f,N_{ab}]\;].
\end{equation}

Note for any random variable $W$ we have that the variance of $W$ is half the variance
of $W-W^*$ where $W^*$ designates an independent copy of $W$. So,
we have
$$VAR[W]=0.5\cdot E[(W-W^*)^2].$$
Let us apply this idea
to \ref{VARfN2}. For this let $N^*_b$ be a variable which conditional on $N_{ab}$ is independent 
 of $N_b$ and has same distribution as $N_b$. Hence, we request that for every $i\leq n$, we have:
$$\mathcal{L}(N^*_b,N_b|N_{ab}=i)=\mathcal{L}(N_b|N_{ab}=i)\otimes\mathcal{L}(N_b^*|N_{ab}=i)$$
and
$$\mathcal{L}(N_b|N_{ab}=i)=\mathcal{L}(N_b^*|N_{ab}=i).$$ 
We also assume that $N^*_b$ is independent of $f(.,.)$.

Then, we have that
\begin{equation}
\label{VARfN3}
VAR[f(N_b,N_{ab})|f,N_{ab}]=0.5\cdot E[\;(f(N_b,N_{ab})-f(N_b^*,N_{ab}))^2|f,N_{ab}]
\end{equation}
Let now $c_2>c_1>0$ be  two constants not depending on $n$. We will see later how we have to select 
these constants. Let $I^n$ be the integer interval
$$I^n:=\left[E[N_b]-c_2\sqrt{n},E[N_b]+c_2\sqrt{n}\right].$$
Let $$G^n_I$$ be the event that
$N_b$ and $N_b^*$ are both in the interval $I^n$.

Let $$G^n_{II}$$ be the event that
$$|N_b-N_b^*|\geq c_1\sqrt{n}.$$
Let $G^n$ be the event:
$$G^n:=G^n_I\cap G^n_{II}.$$
Let $J^n$ denote the integer interval
$$J^n:=[E[N_{ab}]-\sqrt{n},E[N_{ab}]+\sqrt{n}].$$
Let $K^n$ be the event that $N_{ab}$ lies within the interval $J^n$. 

Let $H^n$ be the event that for any $l\in J^n$, we have: 
for any integers $x<y$
in the interval $I^n$ which are  apart by at least
$c_1\sqrt{n}$, the average slope of $f(.,l)$ between $x$ and $y$ is
greater equal than $\Delta/2$, hence:
$$\frac{f(y,l)-f(x,l)}{y-x}\geq \Delta/2.$$
Now, clearly when the events $G^n$, $H^n$  and $K^n$ all hold, then we have
$$|f(N_b,N_{ab})-f(N_b^*,N_{ab})|^2\geq 0.25c_1^2\Delta^2\cdot n.$$
This implies that
\begin{equation}
\label{VARfN4}
E[\;E(f(N_b,N_{ab})-f(N_b^*,N_{ab}))^2|f,N_{ab}]]\geq 
P(G^n\cap H^n\cap K^n)\cdot 0.125 c_1^2\Delta^2\cdot n.
\end{equation}
We can now combine equations \ref{VARfN}, \ref{VARfN2}, \ref{VARfN3} and \ref{VARfN4}, to obtain
\begin{equation}
\label{VARfN5}
VAR[L_n(S)]\geq P(G^n\cap H^n\cap K^n)\cdot 
0.25c_1^2\Delta^2\cdot n.
\end{equation}
and hence
\begin{equation}
\label{VARfN6}
VAR[L_n(S)]\geq (1-P(G^{nc})-P(H^{nc})-P(K^{nc}))\cdot 
0.25c_1^2\Delta^2\cdot n.
\end{equation}
By the Central Limit Theorem, when taking $c_2$ large enough (but not depending
on $n$), we get that the limit $\lim_{n\rightarrow\infty}P(G^n_I)$ gets as close to $1$
as we want. Similarly, looking at Lemma \ref{lemmaGII}, we see that
taking $c_1>0$ small enough (but not depending on $n$), the limit 
$\lim_{n\rightarrow\infty}P(G^n_{II})$ gets also as close to $1$ as we want.
Hence, taking $c_1>0$ small enough and $c_2>0$ large enough, we get the 
the limit for $n\rightarrow\infty$ of $P(G^{nc})$ as close to $0$ as we want.
By Lemma \ref{rafi}, we know that $P(H^{nc})$ goes to $0$ as $n\rightarrow\infty$.
Finally by the Central Limit Theorem, the probability $P(K^{nc})$ converges to
a number bounded away from $1$ as $n\rightarrow\infty$. Applying all of this,
to inequality \ref{VARfN6}, we find that for $c_1>0$ small enough and $c_2>0$
large enough, (but both not depending on $n$), we have: 
there exists a constant $c>0$ not depending on $n$ so that for all $n$ large enough, we have
$$VAR[L_n(S)]\geq cn,$$
as claimed in the lemma.
\end{proof}

\begin{lemma} \label{lemmaGII} It is true that 
$$P(G^n_{II})\rightarrow 2P\left( \mathcal{N}(0,1)\geq \frac{c_1}{\sqrt{2p_b}}\right)$$
as $n\rightarrow\infty$
\end{lemma}

\begin{proof} Let $c>0$ be constant. Let $J^n(c)$ be the interval
$$J^n(c)=[\;E[N_{ab}]-c\sqrt{n},E[N_{ab}]+c\sqrt{n}\;].$$
Let $K^n(c)$ denote the event that $N_{ab}$ is in $J^n(c)$. Note that by Law of Total Probability:
\begin{equation}
\label{totalIII}
P(G^n_{II})=P(G^n_{II}|J^n(c))P(J^n(c))+P(G^n_{III}|J^{nc}(c))P(J^{nc}(c)).
\end{equation}
Now
\begin{equation}
\label{sun}
P(G^n_{II}|J^n(c))=\sum_{k\in J^n(c)}P(G^n_{II}|N_{ab}=k)\cdot P(N_{ab}=k|J^n(c)).
\end{equation}

But  conditioning on $N_{ab}=k$, the variables
$N_b$ and $N_b^*$ become binomial with parameters $p_b/(p_a+p_b)$ and $k$.
Furthermore, $N_b$ and $N_b^*$ are independent of each other
conditional on $N_{ab}=k$. We can hence apply the Central Limit Theorem
and find that conditional on $N_{ab}=k$,
the variable $N_b-N_b^*$ is close to normal with expectation $0$ and
variance $2kq$, where $q:=p_b/(p_a+p_b)$. Hence, by Central Limit Theorem, the probability of $G^n_{II}$,
conditional on $N_{ab}=k$, is approximated by the following probability
$$P\left(|\mathcal{N}(0,2kq)|\geq c_1\sqrt{n} \right)
=2P\left(\mathcal{N}(0,1)\geq \frac{c_1\sqrt{n}}{\sqrt{2kq}}\right).$$
Let us denote by $\epsilon^n_k$ the approximation error, so that
$$\epsilon^n_k:=P(G^n_{II}|N_{ab}=k)-2P\left(\mathcal{N}(0,1)\geq \frac{c_1\sqrt{n}}{\sqrt{2kq}}\right).$$

When $k$ is in $J^n(c)$, then the expression
$$\frac{c_1\sqrt{n}}{\sqrt{2kq}}$$ ranges between
$$a^n_-:=\frac{c_1}{\sqrt{2p_b+2cq/\sqrt{n}}}$$
and
$$a^n_+:=\frac{c_1}{\sqrt{2p_b-2cq/\sqrt{n}}}.$$
From this and Equation \eqref{sun} it follows that
\begin{align}
\label{today}
\sum_{k\in J^n(c)}\epsilon^n_k\cdot& P(N_{ab}=k|J^n(c))
+2P(\mathcal{N}(0,1)\geq a^n_-)
 \leq P(G^n_{III}|J^n(c))\\ \label{today2}
&\leq\sum_{k\in J^n(c)}\epsilon^n_k\cdot P(N_{ab}=k|J^n(c))+2P(\mathcal{N}(0,1)\geq a^n_+)
\end{align}
Assume that $n$ is large enough, (recall that $c>0$ does not depend on $n$),
so that the left most point of $J^n(c)$ is above $n(p_a+p_b)/2$. (How large $n$ needs
be for this depends on $c$). Then, when $k\in J^n(c)$ we have for $n$ large enough,
that $k\geq n(p_a+p_b)/2$. Note that by Berry-Essen inequality we have
that 
$$|\epsilon^n_k|\leq \frac{C^*}{\sqrt{k}}$$
and hence, for all $k\in J^n(c)$ (provided $n$ is large enough),
we find
that
\begin{equation}
\label{C}
|\epsilon^n_k|\leq \frac{C}{\sqrt{n}}
\end{equation}
where $C,C^*>0$ are constants not depending on $n$.
Using  \eqref{C}, we can  rewrite the inequalities given in \eqref{today} and \eqref{today2},
and obtain that for all $n$ large enough we have:
\begin{equation}
\label{today3}
-\frac{C}{\sqrt{n}}
+2P(\mathcal{N}(0,1)\geq a^n_-)
 \leq P(G^n_{III}|J^n(c))\leq
\frac{C}{\sqrt{n}}+2P(\mathcal{N}(0,1)\geq a^n_+).
\end{equation}
When $n\rightarrow\infty$, we have that $a^n_-$ and $a^n_+$ both converge to
$c_1/\sqrt{2p_b}$ and $C/\sqrt{n}$ goes to $0$. Hence, we can apply the Hospital rule
for limits
to the system of inequalities \ref{today3} and find that
\begin{equation}
\label{lim}
P(G^n_{III}|J^n(c))\rightarrow 2P(\mathcal{N}(0,1)\geq \frac{c_1}{\sqrt{2p_b}} )
\end{equation}
as $n\rightarrow\infty$.
Note that by the Central limit theorem, the probability of $J^n(c)$ converges as $n\rightarrow \infty$.
Let $\epsilon(c)$ denote the limit
$$\epsilon(c)=\lim_{n\rightarrow\infty}P(J^{nc}(c)).$$
Taking the lim sup and lim inf of Equation \eqref{totalIII} and using \eqref{lim}
we get 
\begin{align}
\label{coincoin}
 &2P(\mathcal{N}(0,1)\geq \frac{c_1}{\sqrt{2p_b}} )\cdot(1-\epsilon(c))
\leq
\liminf_{n\rightarrow\infty}P(G^n_{II})\leq \limsup_{n\rightarrow\infty}P(G^n_{II})\leq\\
&\label{coincoin2}\leq \;
2P(\mathcal{N}(0,1)\geq \frac{c_1}{\sqrt{2p_b}} )\cdot(1-\epsilon(c))\;+\;\epsilon(c).
\end{align}
Note that the last two inequalities above hold for any $c>0$ not depending on $n$. Furthermore,
 $\epsilon(c)\rightarrow 0$ as $c\rightarrow\infty$. So, letting $c$ go to infinity
we finally find by l'Hospital rule applied to \ref{coincoin} and \ref{coincoin2} that:
$$P(G^n_{II})\rightarrow 2P(\mathcal{N}(0,1)\geq \frac{c_1}{\sqrt{2p_b}} )$$
as $n\rightarrow\infty$.
\end{proof}

\begin{lemma}\label{rafi}
Assume that Inequality \eqref{lalala} holds for $\alpha>0$ not depending on $n$. Then,
we have that
$$P(H^n)\rightarrow 1$$
as $n\rightarrow\infty$.
\end{lemma}

\begin{proof}
Let $H^n_I(k,l)$ be the event that the conditional expected change in optimal alignment
score  when we align $X(k,l)$ with $Y$ is at least $\Delta$. Here we talk
about the change induced by switching  a randomly chosen $a$ into a $b$ in the string $X(k,l)$
or the string $Y(k,l)$.
If $(\tilde{X}(k,l),\tilde{Y}(k,l))$ denotes the randomly modified string pair
$(X(k,l),Y(k,l))$,
then by our definition of $f(.,.)$, we have $f(k+1,l)$ is the optimal alignment score
of $\tilde{X}(k+1,l)$ and $\tilde{Y}(k+1,l)$. Furthermore, $f(k,l)$ denotes the optimal alignment score
of $X(k,l)$ with $Y(k,l)$. Now  formally, the event $H^n(k,l)$ holds when
$$E[f(k+1,l)-f(k,l)|X(k,l),Y]\geq \Delta$$
which is the same as:
$$E[\tilde{L}_n(S)-L_n(S)|X=X(k,l),Y]\geq \Delta$$
or equivalently
\begin{equation}
\label{heinrichIII}
E[\tilde{L}_n(S)-L_n(S)|X,Y,N_b=k,N_{ab}=l]\geq \Delta.
\end{equation}
To understand why the last two inequalities above are equivalent, recall that
the distribution of $(X(k,l),Y(k,l))$ is the same as the distribution of $(X,Y)$ 
conditional on $N_b=k$ and $N_{ab}=l$. 
For the probability of Inequality \eqref{heinrichIII} above, if we would not have
also conditional on $N_b=k$ and $N_{ab}=l$,  we would have the bound
on the right side of \eqref{lalala} available. By how much can a small probability increase
by conditing? Let us take any too events $A$ and $B$. We have
$$P(A|B)=\frac{P(A \cap B)}{P(B)}\leq \frac{P(A)}{P(B)}.$$
So, by conditioning on an event $B$, the probability of any event $A$
increases by at most a factor $1/P(B)$. This leads to
\begin{align}
P(H^{nc}_I(k,l))&=P(E[\tilde{L}_n(S)-L_n(S)|X,Y,N_b=l,N_{ab}=l]< \Delta)\nonumber\\
&\leq\frac{P(E[\tilde{L}_n(S)-L_n(S)|X,Y]< \Delta)}{P(N_b=k,N_{ab}=l)}\label{chair}
\end{align}
Let now $H^n_I$ denote the event:
$$H^n_I=\cap_{k\in I^n,l\in J^n}H^n_I(k,l)$$
so that
\begin{equation}
\label{ionel}
P(H^{nc}_I) \leq \sum_{k\in I^n,l\in J^n} P(H^{nc}_I(k,l)).
\end{equation}
By the assumption of the present lemma that is Equation \eqref{lalala}, we have  the probability
that the following inequality holds
$$E[\tilde{L}_n(S)-L_n(S)|X,Y]< \Delta,$$
is below $n^{-\alpha n}$. Also, by the Local Central Limit Theorem,
we have that there exists a constant $c>0$ not depending on $n$, $k$ or $l$,
so that for all $k\in I^n$ and $l\in I^n$, we have:
$$P(N_b=k,N_{ab}=l)\geq\frac{c}{n}.$$ Applying this and condition \ref{lalala} to inequality \ref{chair},
we find that for $k\in I^n$ and $l\in I^n$, we have:
$$P(H^{nc}_I(k,l))\leq n^{-\alpha n}\cdot n/c=n^{-\alpha n+1}/c.$$
We can now use the last inequality above with 
inequality \ref{ionel}, to find
\begin{equation}
\label{ionel2}
P(H^{nc}_I) \leq 4c_2n^{-\alpha n+2}/c,
\end{equation}
where we used the fact that the number of integer couples $(k,l)$ with
$k\in J^n$ and $l\in I^n$ is $4c_2n$.
Let $M(k,l)$ denote the value:
$$M(k,l)=\sum_{i=0}^{k-1}\left(f(i+1,l)-E[f(i+1,l)|X(i,l),Y(i,l)]\right)\;+f(0,l).$$
Clearly when we hold $l$ fixed, then $M(.,l)$ is a Martingale.

Let $H_{II}^n(x,y,l)$ denote the event that  we have that 
$$|M(y,l)-M(x,l)|\leq 0.5|x-y|\Delta$$
By Hoeffding's Inequality for Martingales, $P(H_{II}^n(x,y,l)$ has high probability, 
\begin{equation}
\label{merde}
P(H_{II}^{nc}(x,y,l))\leq 2\exp(-0.5\Delta^2|x-y|/|S|^2)
\end{equation}
Here $|S|$ denotes the maximum change in value of the scoring function
when we change one letter, 
$$|S|=\max_{c,d,e\in\mathcal{A}^*}|S(c,d)-S(c,e)|.$$
Note that when we change only one letter in a string then the optimal alignment
score changes by at most $|S|$. Since, to obtain $f(k+1,l)$ from $f(k,l)$
we change only one letter, we have that $|f(k+1,l)-f(k,l)|\leq |S|$ always.
This also implies that $|M(k+1,l)-M(k,l)|\leq |S|$ always, which is what we used
to apply Hoeffding inequality.

Now, let
$$H_{II}^n$$ denote the event that $H_{II}^n(x,y,l)$ holds for all $x<y$ with
$|x-y|\geq c_1\sqrt{n}$ and $x,y\in J^n$ and $l\in I^n$.
Then
\begin{equation}
\label{sabi}
P(H_{II}^{nc})\leq \sum_{x,y\in J^n,l\in I^n}P(H_{II}^{nc}(x,y,l))
\end{equation}
where for the sum on the right side of the last equation above is taken
over $|x-y|\geq c_1\sqrt{n}$. The number of triplets $(x,y,l)$ in the sum
on the right side of \ref{sabi} is less than $8c_2^2n^{1.5}$. This bound
together with \eqref{merde} implies
\begin{equation}
\label{sabi2}
P(H_{II}^{nc})\leq  16c_2^2n^{1.5}\exp(-2\Delta^2\sqrt{n}/|S|^2)
\end{equation}
Note that
$$f(k,l)=M(k,l)+\sum_{i=0}^{k-1}E[f(i+1,l)-f(i,l)|X(i,l),Y(i,l)]$$
so that
\begin{equation}
\label{above}
f(y,l)-f(x,l)=M(y,l)-M(x,l)+\sum_{i=x}^{y-1}E[f(i+1,l)-f(i,l)|X(i,l),Y].
\end{equation}
Assume now that $l\in J^n$. Then, when the event $H_I^n$ holds,  the sum of conditional expectations
on the right side of Equation \eqref{above}
 is at least $|y-x|\Delta$. Furthermore when the event $H_{II}^n$ holds
and $|y-x|\geq c_1\sqrt{n}$, then
$$|M(y,l)-M(x,l)|\leq 0.5\Delta|x-y|.$$
It follows looking at \ref{above}, that when both $H^n_I$ and $H^n_{II}$ hold,
and $y-x\geq c_1\sqrt{n}$, that
$$f(y,l)-f(x,l)\geq 0.5|x-y|\Delta$$
This is the condition in the definition of the event $H^n$.
Hence, we have that $H^n_I$ and $H^n_{II}$ together imply $H^n$:
$$H^n_I\cap H^n_{II}\subset H^n$$
and hence
\begin{equation}
\label{HIHII}
P(H^{nc})\leq P(H^{nc}_I)+P(H^{nc}_{II}).
\end{equation}
From the bounds \eqref{sabi2} and \eqref{ionel2}
it follows that  $P(H^{nc}_I)$ and $P(H^{nc}_{II})$ both go to $0$
as $n\rightarrow\infty$. So, because of Equation
\eqref{HIHII}, we find that
$P(H^{nc})$ also goes to $0$ as $n\rightarrow\infty$.
This concludes the proof.
\end{proof}

According to Theorem \ref{samuti}, we have that  $\lambda(S)-\lambda(S-\epsilon T)>0$
implies a positive biased effect
of the random change on the optimal alignment score.
But by lemma\ref{alterhut}, a positive biased effect on the optimal alignment score
implies the fluctuation order: 
\begin{equation}
\label{VAR}
VAR[L_n]=\Theta(n).
\end{equation} Hence,
inequality $\lambda(S)-\lambda(S-\epsilon T)>0$ implies
the fluctuation order given by equation \ref{VAR}.
This is the content of the next theorem:

\begin{theorem}
\label{canary}
Let $S:\mathcal{A}^*\times\mathcal{A}^*\rightarrow\mathbb{R}$
be a scoring function on the finite alphabet $\mathcal{A}$. Let
$T:\mathcal{A}^*\times\mathcal{A}^*\rightarrow\mathbb{R}$ be defined
as 
$$T(a,c)=T(c,a):=S(b,c)-S(a,c)$$
for any $c\in\mathcal{A}^*$ with $c\neq a$ and $T(d,c)=0$ whenever $d\neq a$. \
Furthermore, let $T(a,a)=2(S(b,a)-S(a,a))$. Let $\epsilon>0$.
 If
\begin{equation}
\label{mimi}
\lambda(S)-\lambda(S-\epsilon T)>0,
\end{equation}
then
\begin{equation}\label{fluctfluct}
VAR[L_n(S)]=\Theta(n).
\end{equation}
\end{theorem}

\begin{proof} When 
\begin{equation}\label{hIII}
\lambda(S)-\lambda(S-\epsilon T)>0,
\end{equation}
Theorem \ref{samuti} shows that with high probability the random change has a 
biased effect on the optimal
alignment score. By Lemma \ref{alterhut}, this biased effect
then implies the order of the fluctuation \eqref{fluctfluct}. Let us present further details about this 
argument: Theorem \ref{samuti} implies that Inequality \eqref{h2} follows from \eqref{hIII}.
Let $\delta>0$ be taken as follows,  
$$\delta:=\frac{\lambda(S)-\lambda(S-\epsilon T)}{2\epsilon\cdot p_a},$$
so that Inequality \eqref{h2} becomes 
\begin{equation}\label{whish}
P\left(E[\tilde{L}_n(S)-L_n(S)|X,Y]\geq 
\frac{\lambda(S)-\lambda(S-\epsilon T)}{2\epsilon p_a} \right)\geq 1-n^{-\alpha n}.
\end{equation}

Since, $\lambda(S)-\lambda(S-\epsilon T)$  is strictly positive,
Lemma \ref{alterhut} implies then the desired order of fluctuation,
that is:
$$VAR[L_n(S)]=\Theta(n).$$
We have thus shown that condition \eqref{mimi} implies \eqref{fluctfluct}.
\end{proof}

In many situations the last theorem is very practical tool for verifying the fluctuation order 
\eqref{fluctfluct}. By Montecarlo simulation we can now estimate 
the value for $\lambda(S)$ and $\lambda(S-\epsilon T)$ and test the positivity of the 
quantity $\lambda(S)-\lambda(S-\epsilon T)$ at a given confidence level $\beta$. In case it 
is positive on the chosen confidence level, it follows from Theorem \ref{canary}
that we will also be $\beta$-confident that the fluctuation order \eqref{fluctfluct}
applies. In other words, we check if Inequality \eqref{mimi}
holds at a certain confidence level that will in practice depend on the available 
computational power. In this fashion we can verify for many scoring functions
that $VAR[L_n(S)]=\Theta(n)$ up to a certain confidence level!

\section{Proof of Theorem \ref{samuti} }\label{Theproof}
In order to prove Theorem \ref{samuti},
we need to show that as soon as
$$\lambda(S)-\lambda(S-\epsilon T)>0$$
holds,  we get with high probability
a positive lower bound for the expected effect
of the random change of one letter onto the optimal
alignment score. That lower bound for
$$E[\tilde{L}_n(S)-\tilde{L}_n(S)|X,Y]$$
is as ``close as we want''
(but maybe sligthly below), the following expression, 
$$\frac{\lambda(S)-\lambda(S-\epsilon T)}{\epsilon\cdot p_a}.$$
To prove this, we introduce three events $A^n(S)$, $B^n(S)$
and $C^(\delta)$. We then show in Lemma \ref{combinatoriallemma},
that the three events $A^n(S)$, $B^n(S)$ and $C^n(\delta)$
mutually imply the desired lower bound on the expected change
in optimal alignment score. We then go on to prove that
the events $A^n(S)$, $B^n(S)$ and $C^n(\delta)$ all have 
high probability. This then implies that our lower bound
for the expected change in optimal alignment score
must also hold with high probability. 

So far we traced out a way to prove Theorem \ref{samuti}. Let us now look at 
the details: Let $A^n(S)$ be the event that
$$\frac{L_n(S)}{n}\geq \lambda(S)-\frac{\ln(n)}{\sqrt{n}}.$$

\noindent Let $B^n(S)$ be the event that
$$\frac{L_n(S-\epsilon T)}{n}\leq \lambda(S-\epsilon T)+\frac{\ln(n)}{\sqrt{n}}$$

\noindent For any number $\delta>0$, let $C^n(\delta)$ be the event that
$$\frac{N^n_a}{n}\leq p_a+\frac{\delta\ln n}{\sqrt{n}},$$
where as before $p_a$ is the probability:
$$p_a:=P(X_i=a)=P(Y_i=a).$$

The main combinatorial idea in this paper is given below. It shows
that  the events $A^n(S)$, $B^n(S)$ and $C^n(\delta)$ together
imply the desired lower bound on the expected change of the optimal alignment score
when we change an $a$ into $b$:

\begin{lemma}
\label{combinatoriallemma}Let $\epsilon>0$ be a constant, and assume that 
$$\lambda(S)-\lambda(S-\epsilon T)>0.$$ Let 
$\delta,\delta_1>0$ be any two small constants not depending on $n$.
When $A^n$, $B^n$ and $C^n(\delta_1)$ all hold simultaneously,
then for all $n$ large enough, we have:
$$E[\tilde{L}_n(S)-L_n(S)|X,Y]\geq \frac{\lambda(S)-\lambda(S-T\epsilon)}{\epsilon p_a}-\delta.$$
(How large $n$ needs to be for the above inequality to hold, depends on $\epsilon,\delta,\delta_1,p_a$).
\end{lemma}

\begin{proof}Assume that $A^n(S)$ holds. Then, any optimal alignment
$\pi$ of $X=X_1\ldots X_n$ and $Y=Y_1\ldots Y_n$ satisfies
\begin{equation}
\label{I}
\frac{S_\pi^n}{n}\geq \lambda(S)-\frac{\ln(n)}{\sqrt{n}}
\end{equation}
When $B^n$ holds, then
\begin{equation}
\label{mimi2}
\frac{(S-\epsilon T)_\pi^n}{n}\leq \lambda(S-\epsilon T)+\frac{\ln(n)}{\sqrt{n}}.
\end{equation}
By linearity, however
$$(S-\epsilon T)^n_\pi=S^n_\pi-\epsilon T^n_\pi.$$
The last equation together with inequality \ref{mimi2} leads to:
\begin{equation}
\label{II}
\frac{S_\pi^n-\epsilon T_\pi^n}{n}\leq \lambda(S-\epsilon T)+\frac{\ln(n)}{\sqrt{n}}
\end{equation}
Subtracting Equation \eqref{I} from \eqref{II},
we find 
\begin{equation}
\label{inin}
\frac{\lambda(S)-\lambda(S-\epsilon T)}{\epsilon}-\frac{2\ln(n)}{\epsilon\sqrt{n}}\leq 
\frac{T_\pi^n}{n}
\end{equation}
Now from Equality \eqref{tildeQ*2}, we know that when  changing
a randomly chosen  $a$ into a $b$, the expected effect onto the alignment score
of $\pi$   is  $T^n_\pi/N_a^n$. (Here $N_a^n$ denotes the total number
of $a$'s in the string $X_1X_2\ldots X_n$ and $Y_1\ldots Y_n$ combined). Since $\pi$ is an optimal alignment
according to the scoring function $S$, the expected increase of the alignment score
of $\pi$ is a lower bound for the expected increase of the optimal alignment score.
Hence, the expected increase in optimal
alignment score is at least $T^n_\pi/N_a^n$. (We don't necessarily have equality 
for the change in optimal alignment score, but only a lower bound. The reason
is that we could have another alignment which becomes optimal after we change
a letter.) So, since $T^n_\pi/N_a^n$ is a lower bound for the expected increase
in optimal alignment score, multiplying Inequality \eqref{inin}
by $n/N_a^n$,
we obtain the following lower bound on the expected alignment score change, 
\begin{equation}\label{apple}
E[\tilde{L}_n(S)-L_n(S)|X,Y]\geq 
\frac{n}{N_a^n}\cdot
\left(\frac{\lambda(S)-\lambda(S-\epsilon T)}{\epsilon}-\frac{2\ln(n)}{\epsilon\sqrt{n}} \right)
\end{equation} 
When the event $C^n(\delta_1)$ holds, we find that:

$$\frac{n}{N^n_a}\geq \frac{1}{p_a}\cdot \frac{1}{1+\frac{\delta_1 \ln(n)}{p_a\sqrt{n}}}$$
which we apply to Inequality \eqref{apple} to obtain:
$$E[\tilde{L}_n(S)-L_n(S)|X,Y]\geq
\left(\frac{\lambda(S)-\lambda(S-\epsilon T)}{p_a\epsilon}-\frac{2\ln(n)}{\epsilon p_a\sqrt{n}} \right)
\left( \frac{1}{1+\frac{\delta_1\ln (n)}{p_a\sqrt{n}}}\right)
.$$
From the last inequality above
it follows by continuity, that for all $n$ large enough
$$E[\tilde{L}_n(S)-L_n(S)|X,Y]
\geq
\frac{\lambda(S)-\lambda(S-\epsilon T)}{\epsilon\cdot p_a}-\delta,
$$
as soon as $\delta>0$ does not depend on $n$. We used the fact that $\epsilon>0$, $\delta_1$, $\delta$
and $p_a$ do not depend on $n$. (So how large $n$ needs be depends on $\epsilon$, $\delta$, $\delta_1$
and $p_a$).
\end{proof}

In the next lemma we prove that the event $A^n(S)$ has probability close to $1$,
when $n$ is taken large:

\begin{lemma}For all $n$ large enough, we have that
$$P(A^n(S))\geq 1-n^{-\alpha_1\ln(n)},$$
where $\alpha_1=1/(8|S|^2)$,
and $|S|:=\max_{c,d,e\in \mathcal{A}^*}|S(c,d)-S(c,e)|$.
\end{lemma}

\begin{proof} Note that by Lemma \ref{rateoflambdan}, there exists a constant $c>0$ not depending on $n$,
such that for all $n$ large enough the following inequality
holds:
$$\lambda(S)-\lambda_n(S)\leq \frac{c\sqrt{\ln(n)}}{\sqrt{n}}.
$$
Hence,
\begin{equation}
\label{ambada}
\lambda(S)-\lambda_n(S)-\frac{\ln(n)}{\sqrt{n}}\leq \frac{c\sqrt{\ln(n)}}{\sqrt{n}}-\frac{\ln(n)}{\sqrt{n}}
\leq -\frac{0.5\ln(n)}{\sqrt{n}}
\end{equation}
where the last inequality above holds for $n$ large enough.
Now the event $A^n(S)$ holds exactly when the following inequality is true:
\begin{equation}
\label{oscar}
\frac{L_n(S)}{n}\geq \lambda_n+(\lambda(S)-\lambda_n(S))-\frac{\ln(n)}{\sqrt{n}}.
\end{equation}
The very right side of inequality \ref{ambada}, is an upper bound  for
expression 
$$\lambda(S)-\lambda_n(S)-\frac{\ln(n)}{\sqrt{n}}.$$
 In an inequality giving a lower (non-random) bound
for a random variable, when you replace the lower bound by something bigger,
 the probability (of the inequality) increases. Hence the probability or Inequality \eqref{oscar},
is bigger than the probability
of
\begin{equation}
\label{oscarII}
\frac{L_n(S)}{n}\geq \lambda_n-\frac{0.5\ln(n)}{\sqrt{n}}.
\end{equation} 
This means, that since Inequality \eqref{oscar} is equivalent to the event $A^n(S)$,
that
\begin{equation}\label{AnS}
P(A^n(S))\geq P\left(\frac{L_n(S)}{n}\geq \lambda_n-\frac{0.5\ln(n)}{\sqrt{n}} \right).
\end{equation}
We can now apply McDiarmid's Inequality -- see Lemma \ref{Azuma} -- to the probability on the right-hand 
side of the last inequality to find 
\begin{align}
P\left(\frac{L_n(S)}{n}\geq \lambda_n-\frac{0.5\ln(n)}{\sqrt{n}} \right)
=P\left(L_n(S)-E[L_n(S)]\geq -(2n)\Delta\right)\geq\\\label{ula}
\geq 1-\exp(-(2n)\Delta^2/|S|^2)
\end{align}
where $\Delta=0.25\ln(n)/\sqrt{n}$. We remark that McDiarmid's Inequality is applicable because 
$L_n(S)$ depends on $2n$ i.i.d.\ entries with the property that changing only one entry
affects $L_n(S)$ by at most $|S|$. 

With our definition of $\Delta$ we find that the expression on the very right of
Inequality \eqref{ula} is equal to 
\begin{equation}\label{trivia}
\exp(-(2n)\Delta^2/|S|^2)=\exp(-(\ln(n))^2/8|S|^2)=n^{-\alpha_1 \ln(n)}
\end{equation}
where $\alpha_1=1/(8|S|^2)$. The three equations \eqref{trivia},
\eqref{ula} and \eqref{AnS} jointly imply 
$$P(A^n(S))\geq 1-n^{-\alpha_1 \ln(n)}$$
where $\alpha>0$ is defined by:
$$\alpha_1=\frac{1}{8|S|^2}.$$
\end{proof}

The next lemma shows the high probability of the event $B^n(S)$.

\begin{lemma}
for all $n$ large enough, the following bound holds, 
$$P(B^n(S))\geq 1-n^{-\alpha_2 n}$$
where $\alpha_2:=1/a^2$ and 
$a:=\max_{c,d,e\in \mathcal{A}^*}|S(c,d)-S(c,e)+\epsilon T(c,d)-\epsilon T(c,e)|$.
\end{lemma}

\begin{proof}
A simple subadditivity argument shows that
\begin{equation}
\label{saba}
\lambda_n(S-\epsilon T)\leq \lambda(S-\epsilon T).
\end{equation}
If we change in the definition of the event $B^n(S)$ the upper bound by
something smaller, we  get a lower probability. Hence, because of 
inequality \ref{saba},
we obtain that
\begin{equation}
\label{wie}
P(B^n(S))\geq P\left(\frac{L_n(S-\epsilon T)}{n}\leq \lambda_n(S-\epsilon T)+\frac{\ln(n)}{\sqrt{n}}\right)
\end{equation}
The right side of equation \ref{wie}
is  equal to
\begin{equation}
\label{wie2}
P\left(L_n(S-\epsilon T)-E[L_n(S-\epsilon T)]\leq (2n)\Delta\right)
\end{equation}
where 
$$\Delta=\frac{\ln(n)}{2\sqrt{n}}.$$
We can  apply McDiarmid's Inequality  -- see Lemma \ref{Azuma} -- to the probability given
in \ref{wie2}. We find that \ref{wie2} is greater or equal to
\begin{equation}
\label{wie3}1-\exp(-2(2n)\Delta^2/a^2)=1-\exp(-(\ln(n))^2/a^2)=1-n^{-\ln( n)/a^2}
\end{equation}
where $a^2$ is equal to $1/\alpha_2$. The constant $\alpha_2$ is defined in the statement
of the lemma.

Combining  \eqref{wie3}, \eqref{wie2} and \eqref{wie}, we finally obtain the required inequality 
$$P(B^n(S))\geq 1-n^{-\alpha_2\ln(n)}.$$
\end{proof}

The next lemma shows that the event $C^n(\delta)$ holds with high probability. 

\begin{lemma}
Let $\delta>0$ be a constant. We have that
$$P(C^n(\delta))\geq 1-n^{-2\ln n}.$$
\end{lemma}

\begin{proof}
The event $C^n(\delta)$ is equivalent to
the following inequality:
$$N^n_a-E[N^n_a]\leq \Delta\cdot n$$
where
$$\Delta:=\frac{\ln n}{\sqrt{n}}.$$
by McDiarmid's Inequality, we thus have 
$$P(C^n(\delta))\geq 1-\exp(-2\Delta^2\cdot n)=1-n^{-2\ln n},
$$
as claimed. 
\end{proof}

Let $\delta>0$ not depend on $n$. Lemma \ref{combinatoriallemma} shows that 
when the events $A^n(S)$, $B^n(S)$ and
$C^n(\delta)$ jointly hold, then for $n$ large enough, we have:
\begin{equation}\label{expectedtilde}
E[\tilde{L}_n(S)-L_n(S)|X,Y]\geq \frac{\lambda(S)-\lambda(S-T\epsilon)}{\epsilon p_a}-\delta.
\end{equation}
Hence, Equation \eqref{expectedtilde} holds with high probability, because the events
$A^n(S)$, $B^n(S)$ and $C^n(\delta)$ all hold with high probability. More precisely,
we get:
\begin{align}
\label{refrefref}
P\left(\;E[\tilde{L}_n(S)-L_n(S)|X,Y]\geq \frac{\lambda(S)-\lambda(S-T\epsilon)}{\epsilon p_a}-\delta\right)\geq\\
\geq 1-P(A^{nc}(S))+P(B^{nc}(S))+P(C^{nc}(\delta))
\end{align}
But, by the last three lemma's above, the sum of probabilities
$$P(A^{nc}(S))+P(B^{nc}(S))+P(C^{nc}(\delta))$$
is bounded from above by
$$n^{-\alpha_1\ln(n)}+n^{-\alpha_2\ln(n)}+n^{-2\ln(n)}$$
which for $n$ large enough is bounded from above by
$$n^{-\alpha \ln(n)}$$
where $\alpha>0$ is any constant not depending on $n$
and strictly smaller than $\alpha_1$, $\alpha_2$ and $2$.
So, from Inequality \eqref{refrefref}, we obtain that for all $n$ large enough:
$$
P\left(\;E[\tilde{L}_n(S)-L_n(S)|X,Y]\geq \frac{\lambda(S)-\lambda(S-T\epsilon)}{\epsilon p_a}-\delta\right)\geq
 1-n^{-\alpha n},$$
where $\alpha>0$ does not depend on $n$. This completes the proof of Theorem \ref{samuti}.

\section{The case with the 4 letter genetic alphabet}

\paragraph{Changing a $C$ or $G$ into $A$ or $T$:} We consider here the genetic alphabet $\{A,T,C,G\}$. In this case $A$ and $T$ can mutate easily into
each other. Same thing for $C$ and $G$. But to go from one of these two groups into the other is more difficult.
This implies that when we want to change a letter from the group $\{A,T\}$ into
a letter from the group $\{C,G\}$, we get more heavily punished by the score. Furthermore, 
in the humane genome  the letters $A$ and $T$ have higher frequency than $C$ and $G$. 
We still take $X=X_1X_2\ldots X_n$ and $Y=Y_1Y_2\ldots Y_n$
to be i.i.d. sequences. 
We  consider a model where the probabilities of $A$ and $T$ are equal to each other so that
$$P(X_i=A)=P(Y_i=A)=P(X_i=T)=P(Y_i=T)$$
and the probabilities of $G$ and $C$ are equal to each other:
$$P(X_i=C)=P(Y_i=C)=P(X_i=G)=P(Y_i=G).$$
The random change we consider consists in choosing at random a $C$ or a $G$ and changing
it into a $A$ or a $T$. For this we pick  among all the $C$'s and $G$'s within $X$ and $Y$
one at random with equal probability. Then, we flip a fair coin to decide if
the randomly chosen letter becomes a $A$ or a $T$. Finally we chose the randomly picked
letter into a $A$ or a $T$ depending on the coin. The 
new strings obtained from this one letter change are denoted by $\tilde{X}$ and $\tilde{Y}$.
Hence, there is only one letter changed when going from $XY$ to $\tilde{X}\tilde{Y}$.
This letter is a $C$ or a $G$ which was turned into a $A$ or a $T$.

Again, we denote by $\tilde{L}_n(S)$, the optimal alignment score of $\tilde{X}$ and $\tilde{Y}$
according to $S$, 
$$\tilde{L}_n(S):=\max_\pi S_\pi(\tilde{X},\tilde{Y}),$$
where the maximum above is taken over all alignments with gaps $\pi$ of $\tilde{X}$ with $\tilde{Y}$.
The conditional expected change, as before, is the alignment score of 
a scoring function $T$, which has to be defined sligthly differently
from the previous case. 
We take $T$ as follows, 
for $U$ being equal to $C$ or $G$ and $V\in\{A,C,G,T,g\}$,
 we define first $T_X$, 
$$T_X(U,V):=0.5(S(A,V)-S(U,V))+0.5(S(T,V)-S(U,V)).$$
When $U$ is not equal to $C$ or $G$,
then let $T_X(U,V):=0$.

Similarly, we define $T_Y$ by
$$T_y(V,U):=0.5(S(V,A)-S(V,U))+0.5(S(T,T)-S(V,U)),$$
when $U$ is equal to $C$ or $G$ and $V\in\{A,C,G,T,g\}$. 
Otherwise, we take $T_Y:=0$.
Finally we define $T$ as the sum of $T_X$ and $T_Y$:
$$T=T_X+T_Y.$$
With this definition of $T$,
 the conditional expected change in alignment-score
 $S$ equals the alignment score of $T$ up to a factor.
This is the same principal as the one leading to Equation \eqref{tildeQ*2}.
Hence,  for any alignment $\pi$ of $X$ and $Y$,
the following holds true, 
\begin{equation}
E[S_\pi(\tilde{Y},\tilde{X})-S_\pi(X,Y)|X,Y]=
\label{ENCG}\frac{T_\pi(X,Y)}{N_{C,G}}, 
\end{equation}
where $N_{C,G}$ represents the total number of $C$ and $G$'s present in both $X$ and $Y$.
As usual, $T_\pi(X,Y)$ represents the score of the alignment $\pi$,
when using the scoring function $T$ instead of $S$. Also, $\pi$ is supposed
to align $X=X_1X_2\ldots X_n$ with $Y_1Y_2\ldots Y_n$.

Note that as $n\rightarrow\infty$,
we have
$$\frac{n}{N_{C,G}}\rightarrow \frac{1}{2(p_C+p_G)}=\frac{1}{4p_C}.$$
Hence, in Theorem \ref{samuti} in equation \ref{h2}, we need to replace $p_a$ by
$2(p_C+p_G)$ where $p_c:=P(X_i=C)=P(Y_i=C)$ and $p_G=P(X_i=G)=P(Y_i=G)$.

With these notations, Theorem \ref{samuti} and Lemma \ref{alterhut} remain valid
provided we change $p_a$ by $2(p_C+p_G)$ in equation \ref{h2}. In other words,
in this case also we just have to verify that
$\lambda(S)-\lambda(S-\epsilon T)>0$ to get the variance order
$$VAR[L_n(S)]=\Theta(n).$$
Theorem \ref{samuti} is proved the same way as in the previous case.
So, we leave it to the reader.
The only change is that we start with Equation \eqref{ENCG}, rather than
\eqref{tildeQ*2}. Then one can follow the same steps.
For Lemma \ref{alterhut}, the situation is easier than is is with the change
$a\rightarrow b$ in an alphabet with more than $2$ letters.
Actually, the proof is very similar to the one done in \cite{VARTheta}.
We thus only outline the proof:
when we look at the proof of Lemma \ref{alterhut}, we have two variables:
$N_{ab}$ and $N_b$. In that proof, we condition on $N_{ab}$ and let $N_b$ vary to proof
the fluctuation order. For the genetic alphabet case, we don't need two variables
but only one. So, $N_{A,T}$ will denote the total number of $C$ and $G$'s
counted in both the string $X$ and $Y$. This variable $N_{C,G}$ corresponds
to $N_b$ in the other case). There is no need of another variable
(like $N_{ab}$).  So, we will generate a random sequence of string-pairs:
$$(X(0),Y(0)),(X(1),Y(1)),\ldots,(X(k),Y(k)),\ldots,(X(2n),Y(2n)).$$
The sequences $X(0)$ and $Y(0)$ are i.i.d sequences independent of each other
which contain only the letters $C$ and $G$. Those letters are taken
equiprobable. Then we chose any letter and change it into an $A$ or a $T$.
To decide whether it is $A$ or $T$ we flip a fair coin. We proceed by induction
on $k$: once $(X(k),Y(k))$ is obtained, we chose any $C$ or $G$ in 
$X(k), Y(k)$ and change it to $A$ or $T$. Among all $C$ and $G$'s in both strings
we chose with equal probability. In other words we apply the random change
$\tilde{}$. This means that
our recursive relation is:
$$(X(k+1),Y(k+1))=(\tilde{X}(k),\tilde{Y}(k)).$$
Note that with this definition, the total number of $A$ and $T$'s
in $X(k)$ and $Y(k)$ combined is exactly $k$. Given, that constrain,
all possibilities are equally likely for $(X(k),Y(k))$. This is to
say, that the probability distribution of $(X(k),Y(k))$
is the same as $(X,Y)$ conditional on $N_{A,T}=k$:
$$\mathcal{L}(X(k),Y(k))=\mathcal{L}(X,Y|N_{A,T}=k).$$
So, if we produce the string-pairs $(X(k),Y(k))$ independently of
$N_{A,T}$, then we obtain that
$$(X(N_{A,T}),Y(N_{A,T}))$$
has the same distribution as $(X,Y)$. So, among other, the fluctuation
of the optimal alignment score must be equal as well
\begin{equation}
\label{oesel}
VAR[S(X(N_{A,T}),Y(N_{A,T}))]=VAR[S(X,Y)]=VAR[L_n(S)].
\end{equation}
(Here $S(X(N_{A,T}),Y(N_{A,T})$ denotes the optimal alignment score
of the strings $X(N_{A,T})$ and $Y(N_{A,T})$. Similarly $S(X,Y)$ denotes
the optimal alignment score of $X$ and $Y$.)
so, if we denote
$S(X(k),Y(k))$ by $f(k)$, equation \ref{oesel} becomes
\begin{equation}
\label{Spring}
VAR[f(N_{AT})]=VAR[L_n(S)].
\end{equation}
Now, assume that the random change has typically
a biased effect on the alignment score as given in Equation
\eqref{lalala} in Lemma \ref{alterhut}. We have that $f(k+1)$ is obtained
from $f(k)=S(X(k),Y(k))$ by applying the random change.
So, if \eqref{lalala} holds, that that expected random change typically
should be above $\Delta>0$. So typically,
$$E[f(k+1)-f(k)|X(k),Y(k)]\geq \Delta$$
where $\Delta>0$ does not depend on $k$. 
In other words, $f(.)$ behaves ``like a biased random walk''. And on a certain
scale, has a slope which , with high probability is at least $\Delta$.
But, assume that $g$ is a non-random function with slope at least $\Delta$
Then for any variable $N$, it is shown in \cite{bonettolcs}
that
$$VAR[g(N)]\geq \Delta^2\, VAR[N]$$
So, we can apply this to our case, Take $g$ equal to $f$ and $N$ equal to $N_{AC}$.
We get that when Inequality \eqref{lalala} holds, then
\begin{equation}
\label{mimina}
VAR[f(N_{AT})]\geq \Delta\, VAR[N_{AC}]=\Delta^24ncp_{AC}(1-p_{AC})
\end{equation}
where $c>0$ is a constant not depending on $n$.
Here, the constant $c$ had to be introduced, because
$f$ is random and is not everywhere having a slope of at least $\Delta$
but only with high probability and on a certain scale.
We also used the fact that $N_{AC}$ is a binomial variable with parameters $2n$
and $P(X_i\in\{A,X\})$.
Combining now \eqref{mimina} with \eqref{Spring},
we finally obtain the desired result
$$VAR[L_n(S)]\geq \Delta^24ncp_{AC}(1-p_{AC})$$
and hence
$$VAR[L_n(S)]=\Theta(n).$$

\section{Determining when $\lambda(S)-\lambda(S-\epsilon T)>0$ using simulations}
\label{monte}

Recall that $X_1X_2\ldots X_n$ and
$Y_1Y_2\ldots Y_n$ are two i.i.d.
sequence independent of each other.
Also recall that
$$L_n(R)$$ designates the optimal alignment score of $X_1\ldots X_n$ and
$Y_1\ldots Y_n$ according to the scoring function $R$. Furthermore,
we saw that $L_n(R)/n$ converges to a finite number as $n\rightarrow\infty$
which we denote by
$\lambda_R$, so that
$$\lambda_R:=\lim_{n\rightarrow\infty}\frac{L_n(R)}{n}$$\\
We know by Theorem \ref{canary}, that when 
\begin{equation}
\label{heino}
\lambda(S)-\lambda(S-\epsilon T)>0,
\end{equation}
the fluctuation of the optimal alignment score is linear in $n$,
that is, 
\begin{equation}
\label{teserra}
VAR[L_n(S)]=\Theta(n).
\end{equation}
So,  we can run a Montecarlo simulation, and estimate
the quantity on the left-hand side of \eqref{heino}. If the estimate is positive,
this is an indication that the left side of \ref{heino} is positive too
and that \eqref{teserra} holds. We can even go one step further and actually test on a certain
significance level if inequality \eqref{heino} is satisfied.
If it is on a significance level $\beta>0$, we are then $\beta$-confident
that the order of the fluctuation is as given in inequality \eqref{teserra}.
In this way, we are able to verify up to a certain confidence level
that the fluctuation size of the optimal alignment score is linear
in $n$. We manage to do so for several realistic scoring functions.

To estimate the expression on the right-hand side of \eqref{heino}, we simply use
$(L_n(S)-L_n(S-\epsilon T))/n$. (Note that as $n$ goes to infinity our estimate
goes to $\lambda(S)-\lambda(S-\epsilon T)$.) To do this, we  draw two sequences of length $n$
at random:
$$X=X_1\ldots X_n$$ 
and 
$$Y=Y_1\ldots Y_n.$$

We then take the optimal alignment score
of $X$ and $Y$ according to $S$ which is $L_n(S)$. Next, we calculate the optimal alignment score
of $X$ and $Y$ according to $S-\epsilon T$ which yields $L_n(S-\epsilon T)$. Finally,
we  subtract the
two and divide by $n$ so as to get our estimate of the left side Inequality \eqref{heino},
\begin{equation}
\label{estimate}
\hat{\lambda}(S)-\hat{\lambda}(S-\epsilon T) =\frac{L_n(S)-L_n(S-\epsilon T)}{n}.
\end{equation}
When our estimate is positive, it makes it seem likely that 
Inequality \eqref{heino} is satisfied. We need to ask ourselves however
how big the estimate needs to be, to guarantee that \eqref{heino} holds
up to a high enough confidence level.

When our estimate is positive, we determine at which confidence level
\eqref{heino} holds. Assume that the value reached by our estimate 
is $x$. (So, after one simulation, $x$ designates the numerical value taken
by \eqref{estimate}.)  For the confidence level, we need  an upper bound on the probability
that the estimate reaches the value $x$ if in reality $\lambda_S-\lambda_{S-\epsilon T}$
was negative. The confidence level is then, one minus this probability. 

Let us go through the calculation. First we denote by $E_n$ the following expectation:
$$E_n:=\frac{E[L_n(S)]-E[L_n(S-\epsilon T)]}{n}.$$
We have that
\begin{align}
P&\left(\frac{L_n(S)-L_n(S-\epsilon T)}{n}\geq x\right)\nonumber\\
&\quad=P\left(\frac{L_n(S)-L_n(S-\epsilon T)}{n}-E_n\geq x-E_n\right)\label{hilda1}\\
\label{hilda2}
&\quad\leq P\left(\frac{L_n(S)-L_n(S-\epsilon T)}{n}-E_n\geq x-E_n+(\lambda(S)-\lambda(S-\epsilon T))\right),
\end{align}
where the last inequality above was obtained because we make the assumption that
$\lambda(S)-\lambda(S-\epsilon T)<0$.  
Now,
\begin{equation}
\label{stupidstudent}
 -E_n+(\lambda(S)-\lambda(S-\epsilon T))
=\lambda(S)-\frac{L_n(S)}{n}-\left(\lambda(S-\epsilon T)-\frac{L_n(S-\epsilon T)}{n} \right)
\end{equation}
by subadditivity we have that 
\begin{equation}
\label{little}
\lambda(S)-\frac{L_n(S)}{n}\geq 0.
\end{equation}
In the appendix, Lemma \ref{rateoflambdan} allows us to bound
from above the quantity:
$$\lambda(S-\epsilon T)-\frac{L_n(S-\epsilon T)}{n}$$
by the bound:
\begin{equation}
\label{yesyesyes}
c_n|S-\epsilon T|\cdot\frac{\sqrt{\ln(n)}}{\sqrt{n}},
\end{equation}
where
$$c_n=\sqrt{\frac{2\ln3+2\ln(n+2)}{\ln(n)}}.$$
(Note that we leave out  the term $\frac{2|S|_*}{n}$ which appears in inequality
\ref{difference}. This term is of an order to small to be practically relevant.)
Using now the upper bound \ref{yesyesyes} and inequality \eqref{little} with \eqref{stupidstudent}
in \eqref{hilda1} and \eqref{hilda2}, we finally find
\begin{multline}
\label{carter}
P\left(\frac{L_n(S)-L_n(S-\epsilon T)}{n}\geq x\right)\\\leq
P\left(\frac{L_n(S)-L_n(S-\epsilon T)}{n}-E_n\geq x-c_n|S-\epsilon T|\cdot\frac{\sqrt{\ln(n)}}{\sqrt{n}}\right).
\end{multline}
We can now use Azuma-Hoeffding Inequality (see  Lemma \ref{Azuma} in Appendix)
to bound the probability on the right side of inequality \ref{carter}. As a matter of fact, when we change one of the $2n$
i.i.d. entries (which are $X_1...X_n$ and $Y_1\ldots Y_n$), the term
$$L_n(S)-L_n(S-\epsilon T)$$
changes by at most a quantity
$$|S|+|S-\epsilon T|,$$
where, as before, $|R|$ denotes the msaximum change in aligned letter pair score when one changes on letter
with a scoring function $|R|$,
$$|R|:=\max_{c,d,e\in \mathcal{A}^*}|R(c,d)-R(c,e)|.$$
So,  applying Lemma \ref{Azuma} to the right side expression of \eqref{carter}, we find
\begin{equation}
\label{rara}
P\left(\frac{L_n(S)-L_n(S-\epsilon T)}{n}\geq x\right)\leq \exp(-n\Delta^2/(|S|+|S-\epsilon T|)^2),
\end{equation}
where
$$\Delta=x-c_n|S-\epsilon T|\cdot\frac{\sqrt{\ln(n)}}{\sqrt{n}}).$$
One minus the bound on the right side of \ref{rara} is how confident we are
that $\lambda(S)-\lambda(S-\epsilon T)$ is not negative. Of course, for this to make sense,
we need to to first check that the value of the estimate $x$ is above
$c_n|S-\epsilon T|\cdot\sqrt{\ln(n)}/\sqrt{n}$.\\
In what follows, $S$ refers to the substitution matrix:
$$(S(i,j))_{i,j\in \mathcal{A}},$$
which is obtained from the scoring function $S$. (Basically the matrix $S$, is just a way of writing
the scoring function $S:\mathcal{A}\times \mathcal{A}\rightarrow \mathbb{R}$ in matrix form.) Also,
in all the examples we investigated we took the gap penalty to be the same for all letters:
this means that aligning any letter with a gap has the same score not depending on which letter
gets aligned with the gap. We denote by $\delta$ the gap penalty, that is
$$\delta :=-S(c,G)$$
where the expression on the right side of the above equality in the situation examine numerically in
this paper does not depend on which letter 
$c\in \mathcal{A}$ we consider. ( Recall that  $G$ denotes the symbol used for a gap).\\

Let us quickly explain the situation for which we verified through Montecarlo-simulation that with a high confidence level
$\lambda(S)-\lambda(S-\epsilon T)>0$ for a $\epsilon>0$:
\begin{enumerate}

\item{}The first situation is the same as the first except that we
 we change a $0$ into $1$ in  the sequences $X$ and then another $0$ into $1$ in $Y$.
So the random change consists of two letters changed. This then yields the  matrix $T$
to be
 $$T_2:=
\left(\begin{array}{cc}
-4&2\\
2&0
\end{array}
\right)$$
everything else remains the same.
\item{} Another situation is the DNA-alphabet $\{A,T,C,G\}$. In this case $A$ and $T$ can mutate easily into
each other. Same thing for $C$ and $G$. But to go from one of these two groups into the other is more difficult.
This implies that when we want to change a letter from the group $\{A,T\}$ into
a letter from the group $\{C,G\}$, we get more heavily punished by the score.
This can be seen the default substitution matrix used by Blastz:
$$S_{BLASTZ}=S_{BL}=
\left(
\begin{array}{c|cccc}
  &A   &T   &C   &G\\\hline
A &91  &-31 &-114&-123\\
T &-31 &100 &-125&-114\\
C &-114&-125&100 &-31\\
G &-123&-114&-31 &91
\end{array}
\right)
$$
In humane genome  the letters $A$ and $T$ have higher frequency than $G$ and $C$.
We took $A$ and $T$ together to both have frequency $0.4$ and $G$ and $C$ to each have frequency
 $0.1$.  With these choices and a gap penalty of $800$ we obtained the desired result.
The random change for this is defined as follows:\\
we pick one $C$ or   $G$ in any of the two sequences $X$ and $Y$. That is we consider
all $C$'s and all $G$'s appearing in both $X$ and $Y$ and with equal probability
just chose one such letter. Then we flip a fair coin to decide if 
we change that symbol into a $A$ or a $T$ and then do the change accordingly.
The new strings are denoted by $\tilde{X}$, resp. $\tilde{Y}$. The difference
between $XY$ and $\tilde{X}\tilde{Y}$ is exactly one $C$ or $G$ which got turned into
a $A$ or a $T$.\\
The random-change matrix $T$ in that case is equal to:
$$T_{BLASTZ}=T_{BL}=
\left(
\begin{array}{c|cccc}
  &A   &T   &C   &G\\\hline
A &0  &0 &144&153\\
T &0 &0 &159.5&148.5\\
C &144&159.5&-439&-176\\
G &153&148.5&-176 &-419
\end{array}
\right)
$$
Note that the random change described here tends to increase the score since $C$ and $G$ are likely to be aligned
with $A$ or $T$ since there are more $A$ and $T$'s...
The BLASTZ default gap penalty is $400$, but for significantly 
determining that \ref{heino} holds, we need a higher gap penalty $\delta$ of $1200$.
\end{enumerate}

Let us summarize what we found in our simulations:
\begin{equation*}
\begin{array}{|c|c|c|}
\hline
\text{Case}& \text{I} & \text{II}\\
\hline
\text{Alphabet} & \{0,1\} & \{A,T,C,G\}\\
\hline
P(\cdot) & p_0=0.2,p_1=0.8 & p_A=0.4,p_T=0.4, p_C=0.1,p_G=0.1\\
\hline
S & id_2 & S_{BL}\\
\hline
T & T_2 & T_{BL}\\
\hline
\delta & 6 & 1200\\
\hline
n & 10^5 & 2\times 10^5\\
\hline
\epsilon & 0.5 & 0.9\\
\hline
\frac{L_n}{n} & 0.0634 & 15.197\\
\hline
\text{p-value} & 0.0102 & 2.4\times 10^{-4}\\
\hline
\end{array}
\end{equation*}

In the table above, $L_n$ designates our test statistic, 
\begin{equation*}
L_n=\frac{L_n(S)-L_n(S-\epsilon T)}{n}, 
\end{equation*}
and $\delta$ denotes the gap penalty. Now, the algorithm to 
find the optimal alignment score of two sequences of length $n$
is of order constant times $n^2$. So, our simulation to obtain $L_n$
with $n=100000$ ran overnight. but if one has more time, one could run
longer sequences and get even better results. For example, we use
the actual default matrix for BLASTZ, but then our gap penalty is $1200$
whilst the default is only $400$. In reality, when doing the simulations
with say a gap penalty of $600$ one always get $L_n$ to be positive.
But not positive enough to beat the theoretical our bound for  the difference
between $E[L_n]/n$ and the limit $\lambda(S)-\lambda(S-\epsilon T)$.
Now, there are known methods \cite{Paterson1},\cite{Paterson2}, \cite{martinezlcs}, \cite{lcscurve}, to find confidence bounds for $\lambda(S)$
which are way better than what we use here. (In this paper we simply
simulate two long sequences $X=X_1\ldots X_n$ and$Y=Y_1\ldots Y_n$
and then compute the optimal alignment scores for $S$ and $S-\epsilon T$.
The difference of the scores leads than to $L_n$.) So, using some of 
these advanced methods or running very long simulations, clearly in our opinion
will allow for proving the order 
\begin{equation}
\label{againagain}
VAR[L_n(S)]=Theta(n)
\end{equation} for even ``less extrem''
situations. For example, we expect that if the gap penalty is $600$
instead of $1200$ we still should manage to show \ref{againagain}.
Also, when the probabilities are even less biased, say $0.2,0.2,0.3,0.3$ instead
of $0.1,0.1,0.4.0.4$. Non the less, what we achieve in this article
is already quite remarkable, considering that in the article \cite{},
it takes for  binary-sequences, the probability of $1$ to be below
$10^{-12}$ for the technique to work!! Compare this with the probabilities
in this paper of $P(X_i=1)=0.2,P(X_i=0)=0.8$ for which we are able
to show that \ref{againagain} holds up to a high confidence level! 


\section{Appendix: Large Deviations}\label{appendix}

We denote by $L_S(x_1\dots x_i,y_1\dots y_j)$ the optimal alignment score of the strings
$x_1\ldots x_i$ with $y_1\ldots y_j$ according to the scoring function $S$. Also, recall the definition given in the first section: $L_n(S):=L_S(X_1\dots X_n, Y_1\dots Y_n)$ and $\lambda_n(S):=\expect[L_n(S)]/n$.
Furthermore, recall that  $\lambda_n(S)\rightarrow\lambda(S)$. In this appendix we will show a stronger result that quantifies the convergence rate as being of order $O(\sqrt{\ln n/n})$. For this purpose, we introduce the following notation, 
\begin{align*}
\|S\|_{\delta}&=\max_{c,d,e\in{\mathcal A}^*}\left|S\left(c,d\right)-S\left(c,e\right)\right|,\\
\|S\|_{\infty}&=\max_{c,d\in{\mathcal A}^*}\left|S\left(c,d\right)\right|,
\end{align*}

\begin{lemma}\label{change}
Let $x=x_1\dots x_m$ and $y=y_1\dots y_n$ be two given strings with letters from the alphabet ${\mathcal A}$, and let $S$ be a given scoring function. Let further $\hat{x}\in{\mathcal A}$, and consider two amendments of string $x$, $x^{[i]}=x_1\dots x_{i-1}\,\hat{x}\,x_{i+1}\dots x_m$, obtained by replacing an arbitrary letter $x_i$ by $\hat{x}$, and $x^{[+]}=x_1\dots x_m\,\hat{x}$, obtained by extending $x$ by a letter $\hat{x}$. Then the following hold true, 
\begin{align}
\left|L_S(x^{[i]},y)-L_S(x,y)\right|&\leq\|S\|_{\delta},\label{first claim}\\
\left|L_S(x^{[+]},y)-L_S(x,y)\right|&\leq\|S\|_{\infty}.\label{second claim}
\end{align}
\end{lemma}

\begin{proof}
Let $\pi$ be an optimal alignment of $x$ and $y$, so that $S_{\pi}(x,y)=L_S(x,y)$, and denote the letter with which $x_i$ is aligned under $\pi$ by $a\in{\mathcal A}^*$. Then 
\begin{equation*}
L_S(x^{[i]},y)\geq S_{\pi}(x^{[i]},y)=S_{\pi}(x,y)-S(x_i,a)+S(\hat{x},a)
\geq L_S(x,y)-\|S\|_{\delta}. 
\end{equation*}
Applying the identical argument to an optimal alignment of $x^{[i]}$ and $y$, we obtain the analogous inequality 
\begin{equation*}
L_S(x,y)\geq L_S(x^{[i]},y)-\|S\|_{\delta}, 
\end{equation*}
so that \eqref{first claim} follows. 

For the second claim, let us use an optimal alignment $\pi$ of $x$ and $y$ to construct an alignment $\pi^{[+]}$ of $x^{[+]}$ and $y$ by appending an aligned pair of letters $(\hat{x},G)$, where 
$G$ denotes a gap. Then we have 
\begin{equation*}
L_S(x^{[+]},y)\geq S_{\pi^{[+]}}(x^{[+]},y)=S_{\pi}(x,y)+S(\hat{x},G)\geq L_S(x,y)-\|S\|_{\infty}. 
\end{equation*}
Conversely, we can amend an optimal alignment $\tilde{\pi}^{[+]}$ of $x^{[+]}$ and $y$ to become a valid alignment $\tilde{\pi}$ of $x$ and $y$ by cropping the last pair of aligned letters, $(\hat{x},a)$. We then have 
\begin{equation*}
L_S(x,y)\geq S_{\tilde{\pi}}(x,y)=S_{\tilde{\pi}^{[+]}}(x^{[+]},y)-S(\hat{x},{ a})\geq L_S(x^{[+]},y)-\|S\|_{\infty},
\end{equation*}
thus establishing \eqref{second claim}. 
\end{proof}

\begin{lemma}\label{rateoflambdan}
The convergence of $\lambda_n(S)$ to $\lambda(S)$ is governed by the inequality 
\begin{equation}\label{difference}
\lambda_n(S)\leq\lambda(S)\leq\lambda_n(S)+ c_n\|S\|_{\delta}\frac{\sqrt{\ln n}}{\sqrt{n}}+\frac{2\|S\|_{\infty}}{n},\quad\forall\,n\in\N,
\end{equation}
where 
\begin{equation*}
c_n:=\sqrt{\frac{2\ln3+2\ln(n+2)}{\ln(n)}}.
\end{equation*}
\end{lemma}

Note that $c_n$ tends to $\sqrt{2}$ when $n\rightarrow\infty$, so that it effectively acts as a constant. 

\begin{proof} 
Let $k,n\in\N$, $m=k\times n$, and let $\mathcal{P}_{m,n}$ denote the set of all pairs $(\vec{r},\vec{s})$ of $2k$ dimensional integer vectors $\vec{r}=[\begin{smallmatrix} r_1&\dots &r_{2k}\end{smallmatrix}]^{\T}\in\N_0^{2k}$ and $\vec{s}=[\begin{smallmatrix}s_1&\dots&s_{2k})\end{smallmatrix}]^{\T}\in\N_0^{2k}$ that satisfy $r_i-r_{i-1}+s_i-s_{i-1}\in\{n-1,n,n+1\}$ for $i=1,2,\dots,2k$, as well as $0=r_0\leq r_1\leq\dots\leq r_{2k}=m$ and $0=s_0\leq s_1\leq\dots\leq s_{2k}=m$.

For $(\vec{r},\vec{s})\in\mathcal{P}_{m,n}$, let $L_{m}(S,\vec{r},\vec{s})$ denote the sum of optimal alignment scores
\begin{equation}\label{bienchen sum herum}
L_m(S,\vec{r},\vec{s}):=\sum_{i=1}^{2k} L_S(X_{r_{i-1}+1}\dots X_{r_i},
Y_{s_{i-1}+1}\dots Y_{s_i}).
\end{equation}
Thus, $L_m(S,\vec{r},\vec{r})$ is the optimal alignment score with the additional constraint that $X_{r_{i-1}+1}\dots X_{r_i}$ be aligned with $Y_{s_{i-1}+1}\dots Y_{s_i}$ for $i=1,2,\dots,2k$.

Note that for $L_m(S)/m$ to be larger than $x$, at least one of the $L_m(S,\vec{r},\vec{s})/m$ would have to exceed $x$. The following inequality holds therefore for all $x\in\N$,
\begin{equation}\label{1209}
\prob\left[\frac{L_m(S)}{m}\geq x\right]\leq\sum_{(\vec{r},\vec{s})\in\mathcal{P}_{m,n}}
\prob\left[\frac{L_m(S,\vec{r},\vec{s})}{m}\geq x\right].
\end{equation}

Lemma \ref{change} shows that a change in the value of any one of the $2m$  i.i.d.\ variables $X_1,\dots,X_m,Y_1,\dots, Y_m$ after sampling them -- whilst leaving the values of the remaining variables unchanged -- causes the value of $L_m(S,\vec{s},\vec{r})$ to change by at most $\|S\|_{\delta}$.  Lemma \ref{Azuma} thus implies that for any $\Delta>0$ we have 
\begin{equation}\label{samuel}
P\left[L_m(S,\vec{r},\vec{s})-E\left[L_m(S,\vec{r},\vec{s})\right]\geq m\Delta\right]\leq\exp\left\{-\frac{m\Delta^2}{\|S\|_{\delta}^2}\right\}.
\end{equation}
Furthermore, Lemma \ref{technical} will establish that 
\begin{equation*}
\frac{E\left[L_m(S,\vec{r},\vec{s})\right]}{m}\leq\lambda_n(S)+\frac{2\|S\|_{\infty}}{n},
\end{equation*}
so that we have 
\begin{align*}
\prob\left[\frac{L_m(S,\vec{r},\vec{s})}{m}\geq\lambda_n(S)+\frac{2\|S\|_{\infty}}{n}+\Delta\right]
&\leq\prob\left[L_m(S,\vec{r},\vec{s})-E\left[L_m(S,\vec{r},\vec{s})\right]\geq m\Delta\right]\\
&\stackrel{\eqref{samuel}}{\leq}\exp\left\{-\frac{m\Delta^2}{\|S\|_{\delta}^2}\right\}.
\end{align*}
Substituting this last bound into \eqref{1209} with  $x=\lambda_n(S)+\frac{2\|S\|_{\infty}}{n}+\Delta$, we obtain
\begin{equation*}
\prob\left[\frac{L_m(S)}{m}\geq\lambda_n(S)+\frac{2\|S\|_{\infty}}{n}+\Delta\right]\leq 
\left[3(n+2)\right]^{2k}\exp\left\{- \frac{m\Delta^2}{\|S\|_{\delta}^2}\right\},
\end{equation*}
where we used the observation that $|\mathcal{P}_{m,n}|\leq[3(n+2)]^{2k}$. 

Next, fix a constant $c$ and let $\Delta=c/\sqrt{n}$. Substitution into the last estimate yields 
\begin{equation}\label{veryimportant}
\prob\left[\frac{L_m(S)}{m}\geq\lambda_n(S)+\frac{2\|S\|_{\infty}}{n}+\frac{c}{\sqrt{n}}\right]
\leq\exp\left\{-k\left(\frac{c^2}{\|S\|_{\delta}^2}-d_n^2\right)\right\},
\end{equation} 
where $d_n=\sqrt{2\ln(3)-2\ln(n+2)}$. Setting $z:=c-d_n\|S\|_{\delta}$ and 
\begin{equation*}
Z^m:=\sqrt{n}\left(\frac{L_m(S)}{m}-\lambda_n(S)-\frac{2\|S\|_{\infty}}{n}-\frac{d_n\|S\|_{\delta}}{\sqrt{n}}
\right),
\end{equation*}
\eqref{veryimportant} can be expressed as 
\begin{equation*}
\prob\left[Z^m\geq z\right]\leq\exp\left\{-k\times\frac{z^2+2 z d_n \|S\|_{\delta}}{\|S\|_{\delta}^2}\right\}.
\end{equation*} 
For $z>0$, the right-hand side can be bounded by the quadratic term alone, 
\begin{equation*}
\prob\left[Z^m\geq z\right]\leq\exp\left\{-\frac{k z^2}{\|S\|_{\delta}^2}\right\},\quad\forall\,z\geq 0.
\end{equation*} 
This yields a bound on $\expect[Z^m]$, 
\begin{equation*}
\expect\left[Z^m\right]\leq\int_0^\infty\prob\left[Z^m\geq z\right]\diff z
\leq\int_0^\infty\exp\left\{-\frac{k z^2}{\|S\|_{\delta}^2}\right\}\diff z
=\sqrt{\frac{\pi \|S\|_{\delta}}{k}},
\end{equation*}
and taking $k\rightarrow\infty$, we find 
\begin{equation}\label{limsup}
\limsup_{m\rightarrow\infty}\expect\left[Z^m\right]\leq 0.
\end{equation}

Finally, we have
\begin{align*}
\lambda(S)&=\lim_{m\rightarrow\infty}\expect\left[\frac{L_m(S)}{m}\right]\\
&=\lim_{m\rightarrow\infty}\left(\sqrt{n}\expect\left[Z^m\right]+\lambda_n(S)+\frac{2\|S\|_{\infty}}{n}+\frac{d_n\|S\|_{\delta}}{\sqrt{n}}\right)\\
&\stackrel{\eqref{limsup}}{\leq}\lambda_n(S)+\frac{2\|S\|_{\infty}}{n}+\frac{c_n\|S\|_{\delta}\sqrt{\ln n}}{\sqrt{n}},
\end{align*}
where we used $c_n \sqrt{\ln n}=d_n$. Since $\lambda_n(S)<\lambda(S)$ by subadditivity, this proves the lemma.
\end{proof}

\begin{lemma}[McDiarmid's Inequality \cite{mcDiarmid}]\label{Azuma} 
Let $Z_1,Z_1,\dots,Z_m$ be i.i.d.\ random variables that take values in a set $D$, and let $g:D^m\rightarrow\R$ be a function of $m$ variables with the property that 
\begin{equation*}
\max_{i=1,\dots,m}\sup_{z\in D^m, \hat{z}_i\in D}\left|g(z_1,\dots,z_m)-g(z_1,\dots,\hat{z}_i,\dots,z_m)\right|\leq C. 
\end{equation*}
Thus, changing a single argument of $g$ changes its image by less than a constant $C$. Then the following bounds hold, 
\begin{align*}
\prob\left[g(Z_1,\dots,Z_m)-\expect[g(Z_1,\dots,Z_m)]\geq \epsilon\times m\right]&\leq\exp\left\{-\frac{2\epsilon^2 m}{C^2}\right\},\\
\prob\left[\expect\left[g(Z_1,\dots,Z_m)\right]-g(Z_1,\dots,Z_m)\geq \epsilon\times m\right]&\leq\exp\left\{-\frac{2\epsilon^2 m}{C^2}\right\}.
\end{align*}
\end{lemma}

\begin{proof}
A consequence of the Azuma-Hoeffding Inequality, see \cite{mcDiarmid}. 
\end{proof}

\begin{lemma}\label{technical} Under the notation introduced in Lemma \ref{rateoflambdan} and its proof, it is true that for all $(\vec{r},\vec{s})\in{\mathcal P}_{m,n}$ the following bound applies,
\begin{equation*}
\frac{\expect\left[L_m\left(S,\vec{r},\vec{s}\right)\right]}{m}\leq\lambda_n(S)+\frac{2\|S\|_*}{n}. 
\end{equation*} 
\end{lemma}

\begin{proof}
Assuming first that $r_i-r_{i-1}+s_i-s_{i-1}=n$, we first note that, by the i.i.d.\ nature of the random variables $X_j$ and $Y_k$ and by symmetry of the scoring function $S$, the following random variables are identically distributed, 
\begin{align*}
&L_S\left(X_{r_{i-1}+1}\dots X_{r_i},Y_{s_{i-1}+1}\dots Y_{s_i}\right),\\
&L_S\left(X_{1}\dots X_{r_i-r_{i-1}},Y_{1}\dots Y_{s_i-s_{i-1}}\right),\\
&L_S\left(X_{r_i-r_{i-1}+1}\dots X_{n},Y_{s_i-s_{i-1}+1}\dots Y_{n}\right). 
\end{align*}
Furthermore, it must be true that 
\begin{multline*}
L_S\left(X_{1}\dots X_{r_i-r_{i-1}},Y_{1}\dots Y_{s_i-s_{i-1}}\right)+
L_S\left(X_{r_i-r_{i-1}+1}\dots X_{n},Y_{s_i-s_{i-1}+1}\dots Y_{n}\right)\\
\leq L_S\left(X_{1}\dots X_{n}, Y_{1}\dots Y_{n}\right), 
\end{multline*}
since any alignments of the two pairs of strings in the left-hand side can be concatenated to yield a valid alignment of the pair of strings in the right-hand side. Taking expectations, we find
\begin{equation}\label{oneletter}
2\expect\left[L_S\left(X_{r_{i-1}+1}\dots X_{r_i},Y_{s_{i-1}+1}\dots Y_{s_i}\right)\right]\leq E[L_n].
\end{equation}

Next, allowing $r_i-r_{i-1}+s_i-s_{i-1}$ any value in $\{n-1, n, n+1\}$, this situation is obtained from the previous case by lengthening or shortening at most one of the strings involved by at most one letter. By Lemma \ref{change}, such an amendment cannot change the optimal alignment score by more than $\|S\|_{\infty}$, so that \eqref{oneletter} gives rise to the inequality 
\begin{equation}\label{talon}
\expect\left[L_S\left(X_{r_{i-1}+1}\dots X_{r_i},Y_{s_{i-1}+1}\dots Y_{s_i}\right)\right]\leq\frac{E[L_n]}{2}+\|S\|_{\infty},
\end{equation}
which applies to the general situation. Taking expectations on both sides of \eqref{bienchen sum herum} and substituting \eqref{talon}, we find 
\begin{equation*}
\expect\left[L_m\left(S,\vec{r},\vec{s}\right)\right]\leq\frac{2k\expect\left[L_n(S)\right]}{2}+2k\|S\|_{\infty}.\end{equation*}
Division by $m$ now yields the claim of the lemma.
\end{proof}

\section{Acknowledgments} 

All three authors wish to thank the Engineering and Physical Sciences Research Council (EPSRC) and the Institute of Mathematics and its Applications (IMA) for generous financial support, Endre S\"uli for supporting their small grant application to the IMA, Pembroke College for granting a College Associateship to Saba Amsalu and Heinrich Matzinger, and the Oxford Mathematical Institute for hosting them during their Oxford visit. \\

Heinrich Matzinger whishes to thank the organizers of the XVI EBP (XVI Escola Brazileira de Probabilidade) for inviting him to give a plenary lecture on the topic of this article. He also whishes to thank Professor Fabio Machado
for inviting him to Brazil for collaboration and for introducing him to the field of gene survival models. 

\bibliographystyle{plain}
\bibliography{bio12}
\end{document}